\documentclass{amsart}
\newcommand{\T}{\mathbb{T}}
\newcommand{\D}{\mathbb{D}}

\newcommand{\Z}{\mathbb{Z}}

\newcommand{\R}{\mathbb{R}}
\newcommand{\C}{\mathbb{C}}
\newcommand{\N}{\mathbb{N}}

\usepackage{amsmath}
\usepackage{booktabs}
\usepackage{color}
\usepackage[english]{babel}
\usepackage[latin1]{inputenc}
\usepackage{bbm}
\usepackage{cite}
\usepackage{hyperref}
\usepackage{enumerate}
\usepackage{lineno}
\usepackage{amssymb}

\usepackage{graphics,amsmath,amssymb}
\usepackage{amsthm}
\usepackage{amsfonts}
\usepackage[mathscr]{euscript}
\usepackage{latexsym}
\usepackage{multirow}
\usepackage{afterpage}
\usepackage{caption}

\theoremstyle{plain}
\newtheorem{theorem}{Theorem}[section]
\newtheorem{lemma}[theorem]{Lemma}

\newtheorem{corollary}[theorem]{Corollary}
\newtheorem{prop}[theorem]{Proposition}

\theoremstyle{definition}
\newtheorem{definition}[theorem]{Definition}
\newtheorem{example}[theorem]{Example}

\theoremstyle{remark}

\numberwithin{equation}{section}

\DeclareSymbolFont{bbold}{U}{bbold}{m}{n}
\DeclareSymbolFontAlphabet{\mathbbold}{bbold}

\allowdisplaybreaks

\title{Abelian Squares and Their Progenies}

\author{Charles Burnette \and Chung Wong}
\address{Department of Mathematics, Xavier University of Louisiana, New Orleans, Louisiana 70125-1098, USA}
\email{cburnet2@xula.edu}
\address{Department of Mathematics, County College of Morris, Randolph, NJ 07869-2086, USA}
\email{cwong24@ccm.edu}



\date{\today}

\begin{document}

\maketitle

\begin{abstract}
A polynomial $P \in \C[z_1, \ldots, z_d]$ is strongly $\D^d$-stable if $P$ has no zeroes in the closed unit polydisc $\overline{\D}^d.$ For such a polynomial define its spectral density function to be $\mathcal{S}_P(\mathbf{z}) = \left(P(\mathbf{z})\overline{P(1/\overline{\mathbf{z}})}\right)^{-1}.$ An abelian square is a finite string of the form $ww'$ where $w'$ is a rearrangement of $w.$

We examine a polynomial-valued operator whose spectral density function's Fourier coefficients are all generating functions for combinatorial classes of constrained finite strings over a $d$-character alphabet. These classes generalize the notion of an abelian square, and their associated generating functions are the Fourier coefficients of one, and essentially only one, $L^2(\T^d)$-valued operator. Integral representations, divisibility properties, and recurrent and asymptotic behavior of the coefficients of these generating functions are given as consequences. Tools in the derivations of our asymptotic formulas include a version of Laplace's method for sums over lattice point translations due to Greenhill, Janson, and Ruci\'{n}ski, a version of stationary phase method for oscillatory integrals with complex phase due to Pemantle and Wilson, and various polynomial identities related to powers of modified Bessel functions of the first kind due to Moll and Vignat.
\end{abstract}

\section{Introduction}

For $r = 1, 2, \ldots,$ the $r^{\text{th}}$ \textit{power sum symmetric polynomial} in $d$ variables is
\[p_{r, d}(x_1, \ldots, x_d) = \sum_{k = 1}^{d} x_k^r.\]
\noindent Power sum symmetric polynomials fall within the scope of the theory of symmetric functions and are exposited in various textbooks and monographs such as those of Stanley \cite{Stanley} and Macdonald \cite{Macdonald}. Although power sum symmetric polynomials habitually surface in commutative algebra and representation theory, the approach adopted here is purely analytic.

A polynomial $P \in \C[z_1, \ldots, z_d]$ is said to be \textit{stable} with respect to a region $\Omega \subseteq \C^d,$ or simply $\Omega$-\textit{stable}, if $P$ is zero-free in $\Omega,$ and \textit{strongly} $\Omega$-\textit{stable} if $P$ is zero-free in the closure $\overline{\Omega}.$ This paper is particularly concerned with strongly $\D^d$-stable polynomials, where $\D$ is the complex open unit disc. Given a strongly $\D^d$-stable polynomial $P$ define its \textit{spectral density function} to be
\[\mathcal{S}_P(\mathbf{z}) = \left(P(\mathbf{z})\overline{P\!\left(1/\overline{\mathbf{z}}\right)}\right)^{-1}\]
where $\overline{\mathbf{z}} = \left(\overline{z_1}, \ldots, \overline{z_d}\right)$ and $1/\mathbf{z} = (1/z_1, \ldots, 1/z_d).$ Notice that the restriction of $\mathcal{S}_{P}(\mathbf{z})$ to the unit polycircle $\T^d$ is $|P(\mathbf{z})|^{-2}.$ The strong $\D^d$-stability of $P$ thus ensures that $\mathcal{S}_{P}$ is analytic on a neighborhood of $\T^d.$

Among all positive functions on $\T^d,$ the spectral density functions of strongly $\D^d$-stable polynomials warrant special attention. They are pivotal in framing higher-dimensional analogues of the Bernstein-Szeg\H{o} measure moment problem and in identifying the strictly positive trigonometric polynomials that admit spectral Fej\'{e}r-Riesz factorizations, as discussed in \cite{BW} and \cite{GW}. Attempts to incorporate these problems into a general framework for multivariable orthogonal polynomials has culminated in a wealth of profound mathematics during the 20th century with rich applications in signal processing \cite{Gohberg}, systems theory \cite{FF, FFGK}, and wavelets \cite[Ch. 6]{Daubechies} to name a few.

Our objective here is not to dive deeper into the characteristic phenomena of strongly $\D^d$-stable polynomials. Rather, we wondered whether there is anything combinatorially meaningful about spectral density over the rings of symmetric polynomials. Searching for potential insights, we investigated the polynomials $p(z_1, z_2, z_3) = 1 - xp_{r, 3}(z_1, z_2, z_3)$ for $|x| < 1/3$ together with their spectral density functions $\mathcal{S}_{p}.$ Treating each $\mathcal{S}_{p}$ as a function of a real variable $x$ that maps into $L^2(\T^3)$, one can calculate the Taylor series expansions of the Fourier coefficients $\widehat{\mathcal{S}_{p}}(j, k, l).$ To start, we computed the first few Taylor coefficients of
\[\widehat{\mathcal{S}_{p}}(0, 0, 0) = \frac{1}{8\pi^3}\int\limits_{0}^{2\pi}\int\limits_{0}^{2\pi}\int\limits_{0}^{2\pi} \mathcal{S}_{p}(e^{i\theta_1}, e^{i\theta_2}, e^{i\theta_3})\,\mathrm{d}\theta_1\mathrm{d}\theta_2\mathrm{d}\theta_3\]
\noindent and obtained the following sequence of integers:\\

\begin{center}
\begin{tabular}{c c c c c c c c c c c c c c}
$n$ & 0 & 1 & 2 & 3 & 4 & 5 & 6 & 7 & 8 & 9 & 10 & 11 & 12 \\ \midrule
$\dfrac{\partial_x^n\widehat{\mathcal{S}_{p}}\left.(0, 0, 0)\right|_{x = 0}}{n!}$ & 1 & 0 & 3 & 0 & 15 & 0 & 93 & 0 & 639 & 0 & 4653 & 0 & 35\,169
\end{tabular}
\end{center}\

\noindent Upon inputting the nonzero entries into Neil Sloane's \textit{On-Line Encyclopedia of Integer Sequences} \cite{OEIS}, we received a match in sequence A002893. This goaded us into pinpointing a combinatorial interpretation of these Taylor coefficients. We especially homed in on Jeffrey Shallit's comment that entry A002893 is the number of abelian squares of length $2n$ over a 3-letter alphabet.

Given a nonempty, finite set $\Sigma$ of characters, an \textit{abelian square} over $\Sigma$ is a string in the free monoid $\Sigma^*$ of the form $ww'$ where $w \in \Sigma^*$ and $w'$ is an anagram of $w.$ Six examples of English abelian squares are \texttt{noon}, \texttt{tartar}, \texttt{intestines}, \texttt{reappear}, \texttt{mesosome}, and \texttt{aa}, the last one being both the shortest and alphabetically first abelian square in the English language. The concept of an abelian square was introduced by Erd\H{o}s in \cite{Erdos}. Richmond and Shallit expound on both the exact and asymptotic enumerations of abelian squares in \cite{RS}.

Although this veers in spirit from more traditional investigations of strongly $\D^d$-stable polynomials, focusing on the connection between the Taylor coefficients of $\widehat{\mathcal{S}_{p}}(0, 0, 0)$ and the number of abelian squares over a 3-letter alphabet was a fruitful diversion. In fact, every Fourier coefficient $\widehat{\mathcal{S}_{p}}(j, k, l)$ is encoded with combinatorial data. In Chapter 3 of this paper, we exhibit multidimensional generalizations of $p(z_1, z_2, z_3),$ which we shall refer to as stabilized power sum symmetric polynomials, whose spectral density functions are also ordinary generating functions (abbreviated OGF hereafter) for combinatorial classes of constrained finite strings over a $d$-character alphabet. Various arithmetic and enumerative properties belonging to the coefficients of these OGFs, including recurrent and asymptotic behavior, are deduced in Chapters 4 and 5.

\section{Notational Conventions}

Taking a page from measure theory, we present the following decomposition of integers. If $a \in \Z,$ we define $a^+ = \max(a, 0)$ and $a^- = \max(-a, 0).$ Much like other such decompositions, $a = a^+ - a^-$ and $|a| = a^+ + a^-.$ Also, for any two integers $a$ and $b,$ there is exactly one integer $c$ such that $a - c^+ = b - c^-,$ namely $c = a - b.$ This follows easily from the fact that $a - b = (a - b)^+ - (a - b)^-.$

We shall also avail ourselves of the neatness and brevity of multi-index notation. If $\alpha \in \Z^d,$ then we set
\[\alpha^+ = (\alpha_1^+, \ldots, \alpha_d^+), \hspace{2.5em} \alpha^- = (\alpha_1^-, \ldots, \alpha_d^-), \hspace{2.5em} \text{abs}(\alpha) = (|\alpha_1|, \ldots, |\alpha_d|),\]
\noindent and if $\alpha \in \N_0^d$ (where $\N_0$ is the set of nonnegative integers),
\[|\alpha| = \sum_{k = 1}^{d} \alpha_k, \hspace{2.5em} \alpha! = \prod_{k = 1}^{d} \alpha_k!, \hspace{2.5em} \binom{|\alpha|}{\alpha} = \frac{|\alpha|!}{\alpha!}.\]
\noindent It will also be advantageous to have a notion of divisibility for integer vectors. Given an integer $a,$ we say that $a \mid \alpha$ if $a \mid \alpha_k$ for $k = 1, \ldots, d.$ We will use $\|\, {\scriptscriptstyle \stackrel{\bullet}{{}}}\, \|_p$ to symbolize the standard $p$-norm.

Let $[a : b]$ denote the set of integers between $a$ and $b,$ inclusive. Because the precise lettering of an alphabet is immaterial for our purposes, we take $\Sigma_d = [1 : d]$ throughout the rest of this manuscript. The length of a string $w$ is the number of characters in the string and is denoted by $|w|.$ The \textit{Parikh vector} of a string $w \in \Sigma_d^*$ is given by $\rho(w) = \left(\rho_1(w), \ldots, \rho_d(w)\right),$ where $\rho_j(w)$ denotes the number of occurrences of $j$ within $w.$ Clearly $|w| = \rho_1(w) + \cdots + \rho_d(w).$ (Rohit Parikh, the namesake of this string signature, first implemented these vectors in his work on context-free languages in \cite{Parikh}.)

Lastly, $I_d$ is the $d \times d$ identity matrix, $\mathbf{0}_d$ is the $d$-dimensional zero vector, and we set $\mathbf{1}_d$ to be the vector in $\Z^d$ with each component equal to 1. 

\section{Fourier Analysis of the Spectral Density Functions of Stabilized Power Sum Symmetric Polynomials}

\begin{definition}
The $r^{\text{th}}$ \textit{stabilized power sum symmetric polynomial} in $d$ variables is
\[p_{r, d}^{\star}[x](z_1, \ldots, z_d) = 1 - xp_{r, d}(z_1, \ldots, z_d),\]
\noindent where $x$ is an indeterminate.
\end{definition}

These polynomials are strongly $\D^d$-stable precisely when $|x| < 1/d.$ Corollary 5.7 of \cite{GK-VW} indicates that $p_{r, d}^{\star}[x](z_1, \ldots, z_d)$ admits the determinantal representation
\begin{equation}
p_{r, d}^{\star}[x](z_1, \ldots, z_d) = \det\!\left(I_d - xJ_d\text{diag}(z_1^r, \ldots, z_d^r)\right)
\end{equation}
\noindent where $J_d$ is the $d \times d$ all-ones matrix. (Grinshpan et al. \cite{GK-VW} proved more generally that every irreducible polynomial $P \in \C[z_1, \ldots, z_d]$ that is zero-free on the closure of a matrix unit polyball with $P(0) = 1$ admits some determinantal representation $P(\mathbf{z}) = \det(I_d - K\text{diag}(z_1, \ldots, z_d)),$ where $K$ is a strictly contractive matrix.) Under the additional assumptions that $x \in \R$ and $\mathbf{z} \in \T^d,$ we can write
\begin{equation}
\overline{p_{r, d}^{\star}[x](z_1, \ldots, z_d)} = \det\!\left(I_d - xJ_d\text{diag}(\overline{z_1}^r, \ldots, \overline{z_d}^r)\right),
\end{equation}
\noindent and so the restriction of the spectral density function $\mathcal{S}_{p_{r, d}^{\star}[x]}$ to $\T^d$ equals
\begin{align} \label{spd}
\frac{1}{\left|p_{r, d}^{\star}[x](z_1, \ldots, z_d)\right|^2} &= \det\!\left(\begin{array}{cc} I_d - xJ_d\text{diag}(z_1^r, \ldots, z_d^r) & 0 \\ 0 & I_d - xJ_d\text{diag}(\overline{z_1^r}, \ldots, \overline{z_d^r}) \end{array}\right)^{-1} \nonumber
\\ & = \det\!\left(I_{2d} - x\!\left[\begin{array}{cc} J_d\text{diag}(z_1^r, \ldots, z_d^r) & 0 \\ 0 & J_d\text{diag}(\overline{z_1}^r, \ldots, \overline{z_d}^r) \end{array}\right]\right)^{-1}.
\end{align}
\noindent If we let $A_d(\mathbf{z}) = (a_{i,j})_{2d \times 2d}$ denote the block diagonal matrix
\[\left[\begin{array}{cc} J_d\text{diag}(z_1^r, \ldots, z_d^r) & 0 \\ 0 & J_d\text{diag}(\overline{z_1}^r, \ldots, \overline{z_d}^r) \end{array}\right],\]
\noindent then the MacMahon Master Theorem gives us that
\begin{equation} \label{MMT}
\frac{1}{\left|p_{r, d}^{\star}[x](z_1, \ldots, z_d)\right|^2} = \sum_{n = 0}^{\infty} \left[\sum_{\substack{(n_1, \ldots, n_{2d}) \in \N_0^d, \\ n_1 + \cdots + n_{2d} = n}} G(n_1, \ldots, n_{2d})\right]x^n
\end{equation}
\noindent where, because of the way $A_d(\mathbf{z})$ is partitioned,
\begin{align} \label{MMTcoefficients}
G(n_1, \ldots, n_{2d}) &= [y_1^{n_1}\cdots y_{2d}^{n_{2d}}]\prod_{i = 1}^{2d} (a_{i,1}y_1 + \cdots + a_{i,2d}y_{2d})^{n_i} \nonumber
\\ & = \binom{n_1 + \cdots + n_d}{n_1, \ldots, n_d}\binom{n_{d + 1} + \cdots + n_{2d}}{n_{d + 1}, \ldots, n_{2d}}\prod_{i = 1}^{d} z_i^{rn_i}\overline{z_i}^{rn_{i + d}}.
\end{align}
\noindent For a given $\mathbf{z} \in \T^d,$ series (\ref{MMT}) converges absolutely for $|x| < 1/\rho(A_d(\mathbf{z})),$ where $\rho(A)$ is the spectral radius of $A.$

Now if we were to calculate an individual Fourier coefficient of (\ref{MMT}), a subset of the $G(n_1, \ldots, n_{2d})$ terms in the resulting formal power series will vanish. In this chapter, we impart enumerative meaning to those terms in the summation that survive. More specifically, the Fourier coefficients of $\mathcal{S}_{p_{r, d}^{\star}[x]}$ are all OGFs for combinatorial classes of constrained finite strings.

\subsection{Offset Words}

Let us first forge the combinatorial classes enumerated by $\widehat{\mathcal{S}_{p_{r, d}^{\star}[x]}}(\xi)$ for $\xi \in \Z^d.$

\begin{definition}
For $(n, \xi) \in \N_0 \times \Z^d,$ an $n^{\text{th}}$ \textit{order word pair offset by} $\xi$ is an ordered pair of words $w, w' \in \Sigma_d^*$ such that the length of $ww'$ is $2n + \|\xi\|_1$ and $\rho(w) - \rho(w') = \xi.$ We denote the set of $n^{\text{th}}$ order word pairs offset by $\xi$ by $\mathcal{W}_{(n, \xi)},$ the set of all word pairs offset by $\xi$ (regardless of order) by $\mathcal{W}_{\xi},$ and let $w_{(n, \xi)} = \left|\mathcal{W}_{(n, \xi)}\right|.$
\end{definition}

Intuitively $\xi$ signals how removed a particular concatenation $ww'$ is from producing an abelian square while the order $n$ tells how many characters $w$ and $w'$ have in common. (See Definition \ref{componentwisemin} and Theorem \ref{composition} below.) Evidently an $n^{\text{th}}$ order word pair offset by $\mathbf{0}_d$ yields an abelian square of length $2n$ since a string $ww'$ is an abelian square if and only if $\rho(w) = \rho(w').$

\begin{example}
Consider the alphabet $\Sigma_3 = \{1, 2, 3\}.$ Let $w = 13$ and $w' = 12.$ Then $(w, w')$ is a first order word pair offset by $(0, -1, 1).$ The string 1312 is also the concatenation of

\begin{itemize}
\item a zeroth order word pair offset by $(-2, -1, -1)$ with $w = \varepsilon,$ the empty string, and $w' = 1312,$

\item a first order word pair offset by $(0, -1, -1)$ with $w = 1$ and $w' = 312,$

\item a zeroth order word pair offset by $(2, -1, 1)$ with $w = 131$ and $w' = 2,$

\item a zeroth order word pair offset by $(2, 1, 1)$ with $w = 1312$ and $w' = \varepsilon.$
\end{itemize}
\end{example}

\begin{prop} \label{nondisjoint}
Every string $w \in \Sigma_d^*$ emanates from $|w| + 1$ of the sets $\mathcal{W}_{(n, \xi)}.$
\end{prop}

\begin{proof}
Suppose $|w| = N$ and fix $k \in [0 : N].$ There is only one way to write $w$ as the concatenation of two strings $w_k$ and $w_{N - k}'$ such that $|w_k| = k$ and $|w_{N - k}'| = N - k.$ In setting $\xi_k = \rho(w_k) - \rho(w_{N - k}'),$ we note that $N - \|\xi_k\|_1$ is even because
\begin{align*}
N - \|\xi_k\|_1 &= N - \sum_{j = 1}^{d} \left|\rho_j(w_k) - \rho_j(w_{N - k}')\right|
\\ &\equiv N - \sum_{j = 1}^{d} \left(\rho_j(w_k) - \rho_j(w_{N - k}')\right) (\text{mod}\, 2) = 2\sum_{j = 1}^{d} \rho_j(w_{N - k}')\, (\text{mod}\, 2).
\end{align*}
\noindent Furthermore
\[\sum_{j = 1}^{d} \left|\rho_j(w_k) - \rho_j(w_{N - k}')\right| \leq \sum_{j = 1}^{d} \left(\left|\rho_j(w_k)\right| + \left|\rho_j(w_{N - k}')\right|\right) = |w|,\]
\noindent and so $\frac{1}{2}(N - \|\xi\|_k)$ is a nonnegative integer. Therefore $w \in \mathcal{W}_{(n_k, \xi_k)}$ with $n_k = \frac{1}{2}(N - \|\xi_k\|_1).$ It is clear that different $k$ yield different $\xi_k$ and that there are only $N + 1$ ways to factor $w$ into two consecutive substring blocks.
\end{proof}

Except for the empty string, which is solely an abelian square, every string in $\Sigma_d^*$ can be associated with multiple offset word pairs. This may seem to undermine the utility of Definition 3.2, but offset word pairs are a natural generalization of abelian squares in the following sense: the OGFs for $w_{(n, \xi)}$ are the Fourier coefficients of one, and essentially only one, operator from $\R$ into $L^2(\T^d).$ We will demonstrate this later in the chapter.

We will now count the number of $n^{\text{th}}$ order word pairs offset by $\xi.$  To elucidate the counting argument, let us first introduce the concept of mutuality.

\begin{definition} \label{componentwisemin}
The \textit{mutuality} $\nu(w, w')$ of two strings $w, w' \in \Sigma_d^*$ is the component-wise minimum of $\rho(w)$ and $\rho(w'),$
\end{definition}

\begin{lemma} \label{mutuality}
For any $w, w' \in \Sigma_d^*,$
\begin{equation}
\nu(w, w') = \rho(w) - (\rho(w) - \rho(w'))^+ = \rho(w') - (\rho(w) - \rho(w')^-.
\end{equation}
\end{lemma}

\begin{proof}
The rightmost equality follows immediately from the integer decomposition defined in Chapter 2. From the fact that for any two numbers $x$ and $y,$ $\min\{x, y\} = (x + y - |x - y|)/2,$ we see that
\begin{align*}
\nu(w, w') &= \frac{1}{2}(\rho(w) + \rho(w') - \text{abs}(\rho(w) - \rho(w')))
\\ &= \frac{1}{2}(\rho(w) + \rho(w') - ((\rho(w) - \rho(w'))^+ + (\rho(w) - \rho(w'))^-))
\\ &= \frac{1}{2}\!\left(\rho(w) - (\rho(w) - \rho(w'))^+\right) + \frac{1}{2}\!\left(\rho(w') - (\rho(w) - \rho(w')^-\right).
\end{align*}
\noindent The result now follows.
\end{proof}

So the concatenation $ww'$ belongs to $\mathcal{W}_{\xi}$ if and only if $\rho(w) = \nu(w, w') + \xi^+$ and $\rho(w') = \nu(w, w') + \xi^-.$ Thus, for each $j \in \Sigma_d,$ the words $w$ and $w'$ contain, respectively, at least $\xi_j^+$ and $\xi_j^-$ instances of $j.$

\begin{lemma} \label{composition}
If $(w,w') \in \mathcal{W}_{(n, \xi)},$ then $\nu(w, w')$ is a weak composition of $n.$
\end{lemma}

\begin{proof}
Since $|ww'| = 2n + \|\xi\|_1,$ it follows from Proposition \ref{nondisjoint} and Lemma \ref{mutuality} that
\[|\nu(w, w')| = \frac{1}{2}\!\left|\rho(w) + \rho(w') - \text{abs}(\xi)\right| = \frac{1}{2}\sum_{j = 1}^{d} [\rho_j(ww') - |\xi_j|] = n.\]
\end{proof}

\begin{theorem}
For all $n \in \N_0$ and $\xi \in \Z^d,$
\begin{equation} \label{enumeration}
w_{(n, \xi)} = \sum_{\substack{\nu \in \N_0^d, \\ |\nu| = n}} \binom{|\nu + \xi^+|}{\nu + \xi^+}\binom{|\nu + \xi^-|}{\nu + \xi^-}.
\end{equation}
\end{theorem}

\begin{proof}
Given a weak composition $\nu$ of $n$ into $d$ parts, the number of ways to build two strings $w$ and $w'$ so that $\rho(w) = \nu + \xi^+$ and $\rho(w') = \nu + \xi^-$ is $\binom{|\nu + \xi^+|}{\nu + \xi^+}\binom{|\nu + \xi^-|}{\nu + \xi^-}.$ Lemmas \ref{mutuality} and \ref{composition} indicate that summing this across all weak compositions of $n$ into $d$ parts gives $w_{(n, \xi)}.$
\end{proof}

\subsection{The Fourier Coefficients as Generating Functions}

We now return our attention to the polynomial $p_{1, d}^{\star}[x](z_1, \ldots, z_d)$ with $x \in \R$ and $\mathbf{z} \in \T^d.$ Since $\mathcal{S}_{p_{1, d}^{\star}[x]}|_{\T^d}$ can be presented as a power series in each variable $z_j,$ it converges uniformly on every compact subset of $(-1/d, 1/d) \times \T^d.$ Of course, expressing $\mathcal{S}_{p_{1, d}^{\star}[x]}|_{\T^d}$ as a power series in the $z_j$ requires rearrangement of the terms. This is justified since (\ref{MMT}) holds on all of $(-1/d, 1/d) \times \T^d,$ rendering the series absolutely convergent.

Let us now compute the Fourier coefficients of $\mathcal{S}_{p_{1, d}^{\star}[x]}:$
\begin{equation}
\widehat{\mathcal{S}_{p_{1, d}^{\star}[x]}}(\xi) = \frac{1}{(2\pi)^d}\int\limits_{[0, 2\pi]^d} e^{-i\xi^{\textsf{T}}\theta}\mathcal{S}_{p_{1, d}^{\star}[x]}(e^{i\theta_1}, \ldots, e^{i\theta_d})\,\mathrm{d}\theta.
\end{equation}
\noindent Here $\xi^{\textsf{T}}$ is the transpose of $\xi,$ $\theta = (\theta_1, \ldots, \theta_d),$ and the integral is relative to the completion of the $d$-fold product of the Lebesgue measure on $\R.$ By (\ref{MMT}), $\widehat{\mathcal{S}_{p_{1, d}^{\star}[x]}}(\xi)$ can be rewritten as
\[\frac{1}{(2\pi)^d}\!\sum_{n = 0}^{\infty}\!\left[\int\limits_{[0, 2\pi]^d}\! e^{-i\xi^{\textsf{T}}\theta}\!\sum_{\substack{(n_1, \ldots, n_{2d}) \in \N_0^d, \\ n_1 + \cdots + n_{2d} = n}}\!\binom{n_1 + \cdots + n_d}{n_1, \ldots, n_d}\!\binom{n_{d + 1} + \cdots + n_{2d}}{n_{d + 1}, \ldots, n_{2d}}\!\prod_{i = 1}^{d}\! e^{i(n_i - n_{i + d})\theta_i}\mathrm{d}\theta\!\right]\!x^n\]
\begin{equation} \label{spdfourier}
= \frac{1}{(2\pi)^d}\sum_{n = 0}^{\infty}\left[\int\limits_{[0, 2\pi]^d}\sum_{k = 0}^{n} \sum_{\substack{\kappa, \kappa' \in \N_0^d, \\ |\kappa| = k, \\ |\kappa'| = n - k}} \binom{|\kappa|}{\kappa}\binom{|\kappa'|}{\kappa'}e^{i(\kappa - \kappa' - \xi)^{\textsf{T}}\theta}\mathrm{d}\theta\right]\!x^n.
\end{equation}
\noindent where the uniform convergence of $\mathcal{S}_{p_{1, d}^{\star}[x]}$ allows us to interchange the summation and the multiple integral with impunity. Fubini's theorem also grants us treatment of these Fourier coefficients as iterated integrals. So since $\int_{0}^{2\pi} e^{ia\theta}\,\mathrm{d}\theta$ is $2\pi$ when $a = 0$ and $0$ if $a \in \Z - \{0\},$ the only terms in the integrand of (\ref{spdfourier}) that remain are those for which $\kappa - \kappa' - \xi = \mathbf{0}_d,$ which implies that $\kappa - \xi^+ = \kappa' - \xi^-.$ Hence, in setting $\nu = \kappa - \xi^+,$ we can simplify the power series for $\widehat{\mathcal{S}_{p[x]_{1, d}^*}}(\xi)$ to
\begin{equation} \label{ogfbylength}
\sum_{n = 0}^{\infty} \left[\sum_{\substack{\nu \in \N_0^d, \\ |\nu| = n}} \binom{|\nu + \xi^+|}{\nu + \xi^+}\binom{|\nu + \xi^-|}{\nu + \xi^-}\right]\!x^{2n + \|\xi\|_1}.
\end{equation}
\noindent Replace $x$ with $\sqrt{x}$ in (\ref{ogfbylength}) and then normalize to conclude the following.

\begin{theorem} \label{ogf}
The OGF for $\mathcal{W}_{\xi},$ with $\xi \in \Z^d,$ counted according to the order $n,$ is
\begin{equation}
W_{\xi}(x) = \frac{1}{x^{\frac{1}{2}\|\xi\|_1}}\widehat{\mathcal{S}_{p_{1, d}^{\star}[\sqrt{x}]}}(\xi).
\end{equation}
\noindent The OGF for $\mathcal{W}_{\xi},$ counted according to the length $2n + \|\xi\|_1,$ is just $\widehat{\mathcal{S}_{p_{1, d}^{\star}[x]}}(\xi).$
\end{theorem}

The Fourier coefficients of $\mathcal{S}_{p_{r, d}^{\star}[x]}$ can be calculated in the same manner by noticing that $p_{r, d}^{\star}[x](z_1, \ldots, z_d) = p_{1, d}^{\star}[x](z_1^r, \ldots, z_d^r).$ Hence
\begin{align} \label{spdrpssp}
\widehat{\mathcal{S}_{p_{r, d}^{\star}[x]}}(\xi) &= \frac{1}{(2\pi)^d}\int\limits_{[0, 2\pi]^d} e^{-i\xi^{\textsf{T}}\theta}\mathcal{S}_{p_{1, d}^{\star}[x]}(e^{ir\theta_1}, \ldots, e^{ir\theta_d})\,\mathrm{d}\theta \nonumber
\\ &= \frac{1}{(2\pi)^d}\sum_{n = 0}^{\infty}\left[\int\limits_{[0, 2\pi]^d}\sum_{k = 0}^{n} \sum_{\substack{\kappa, \kappa' \in \N_0^d, \\ |\kappa| = k, \\ |\kappa'| = n - k}} \binom{|\kappa|}{\kappa}\binom{|\kappa'|}{\kappa'}e^{i(r\kappa - r\kappa' - \xi)^{\textsf{T}}\theta}\mathrm{d}\theta\right]\!x^n
\end{align}
\noindent for $|x| < 1/d.$ Like before, the only terms in the integrand of (\ref{spdrpssp}) that contribute to the sum are those for which $r\kappa - r\kappa' - \xi = \mathbf{0}_d.$ So $\widehat{\mathcal{S}_{p_{r, d}^{\star}[x]}}(\xi) = 0$ unless $r \mid \xi,$ in which case we equivalently require $\kappa - \kappa' - r^{-1}\xi = \mathbf{0}_d.$ The ensuing computation and analytic technicalities involved now match those of $\widehat{\mathcal{S}_{p_{1, d}^{\star}[x]}}(r^{-1}\xi).$

\begin{corollary}
For $\xi \in \Z^d,$
\begin{equation}
\widehat{\mathcal{S}_{p_{r, d}^{\star}[x]}}(\xi) = \widehat{\mathcal{S}_{p_{1, d}^{\star}[x]}}(r^{-1}\xi)\mathbbold{1}_{r\, \mid\, \xi},
\end{equation}
\noindent where $\mathbbold{1}_{r\, \mid\, \xi} = 1$ if $r \mid \xi$ and $0$ otherwise.
\end{corollary}

Consequently, $x^{-\frac{1}{2}\|r^{-1}\xi\|_1}\widehat{\mathcal{S}_{p_{r, d}^{\star}[\sqrt{x}]}}(\xi) = W_{r^{-1}\xi}(x)$ is the OGF for $w_{(n, r^{-1}\xi)},$ provided that $r \mid \xi.$ For this reason, we will henceforth analyze $\mathcal{S}_{p_{1, d}^{\star}[x]}$ exclusively; the combinatorics behind $\widehat{\mathcal{S}_{p_{r, d}^{\star}[x]}}(\xi)$ for $r > 1$ is redundant.

\section{The Integer Sequence $w_{(n, \xi)}$}

\subsection{Combinatorial Meaning of the Fourier Inversion} We begin this section by deriving a multiple integral representation of $w_{(n, \xi)}.$

\begin{corollary}
If $n \in \N_0$ and $\xi \in \Z^d,$ then
\begin{align} \label{integralrep}
\displaystyle w_{(n, \xi)} &= \frac{1}{(2\pi)^d}\int\limits_{[0, 2\pi]^d} \left[e^{-i\xi^{\textsf{T}}\theta}\left(\sum_{k = 1}^{d} e^{i \theta_k}\right)^{|\xi^+|}\!\left(\sum_{k = 1}^{d} e^{-i \theta_k}\right)^{|\xi^-|}\!\left|\sum_{k = 1}^{d} e^{i \theta_k}\right|^{2n}\right]\mathrm{d}\theta.
\end{align}
\end{corollary}

\begin{proof}
According to Theorem \ref{ogf},
\begin{equation}
w_{(n, \xi)} = [x^{2n + \|\xi\|_1}]\widehat{\mathcal{S}_{p_{1, d}^{\star}[x]}}(\xi).
\end{equation}
\noindent Now work backwards from (\ref{spdfourier}). Bearing in mind that the $k = n + |\xi^+|$ term of the sum on the right-hand side of (\ref{spdfourier}) is the only one that contributes to the integral, we see that
\begin{align*}
w_{(n, \xi)} &= \frac{1}{(2\pi)^d}\int\limits_{[0, 2\pi]^d} e^{-i\xi^{\textsf{T}}\theta}\!\sum_{\substack{\kappa, \kappa' \in \N_0^d, \\ |\kappa| = n + \xi^+, \\ |\kappa'| = n + \xi^-}}\binom{|\kappa|}{\kappa}\binom{|\kappa'|}{\kappa'}\prod_{i = 1}^{d} e^{i(\kappa - \kappa' - \xi)^{\mathsf{T}}\theta}\mathrm{d}\theta
\\ &= \frac{1}{(2\pi)^d}\!\int\limits_{[0, 2\pi]^d}\! e^{-i\xi^{\textsf{T}}\theta}\!\left(e^{i\theta_1} +  \cdots + e^{i\theta_d}\right)^{n + |\xi^+|}\!\left(e^{-i\theta_1} + \cdots + e^{-i\theta_d}\right)^{n + |\xi^-|}\mathrm{d}\theta.
\end{align*}
\noindent Formula (\ref{integralrep}) now follows by virtue of the fact that
\[e^{-i\theta_1} + \cdots + e^{-i\theta_d} = \overline{e^{i\theta_1} + \cdots + e^{i\theta_d}}.\]
\end{proof}

In particular, if we set $\xi = \mathbf{0}_{d}$ and perform a change of variables, then we recover the identity
\[\sum_{a_1+\cdots+a_n = k} \binom{k}{a_1,\ldots,a_n}^2 = \int\limits_{[0,1]^n} \left|\sum_{k = 1}^{n} e^{2\pi i x_k}\right|^{2k}\mathrm{d}\theta\]
that was derived in \cite{BNSW}.

\begin{theorem} \label{fourierinversion}
For all $x \in (-1/d, 1/d),$
\[\sum_{\xi \in \Z^d} \widehat{\mathcal{S}_{p_{1, d}^{\star}[x]}}(\xi) = \frac{1}{(1 - dx)^2}.\]
\end{theorem}

Due to the Fourier inversion theorem,
\begin{equation}
\sum_{\xi \in \Z^d} \widehat{\mathcal{S}_{p_{1, d}^{\star}[x]}}(\xi) = \mathcal{S}_{p_{1, d}^{\star}[x]}(\mathbf{1}_d) = \frac{1}{(1 - dx)^2}.
\end{equation}
Here, though, is a different, substantially more informative explanation.

\begin{proof}
By construction, $\sum_{\xi \in \Z^d} \widehat{\mathcal{S}_{p_{1, d}^{\star}[x]}}(\xi)$ is an OGF for the class of ordered pairs of words from $\Sigma_d^{*}.$ In the vein of Section I.4 of \cite{FS}, this class has combinatorial specification $\textsc{Seq}(\Sigma_d) \times \textsc{Seq}(\Sigma_d)$ and thus an OGF of $\frac{1}{(1 - dx)^2}.$
\end{proof}

Equating the coefficients of $x^{\ell}$ for both of the formal power series in Theorem \ref{fourierinversion} leads us to the following corollary.

\begin{corollary}
If $\ell \in \N_0,$ then
\[\sum_{\substack{n \in \N_0, \xi \in \Z^d, \\ 2n + \|\xi\|_1 = \ell}} w_{(n, \xi)} = (\ell + 1)d^{\ell}.\]
\end{corollary}

We can certainly use Corollary \ref{integralrep} to prove this, but, as always, a bijective proof is more enlightening. The left-hand sum is the number of all ordered pairs of words from $\Sigma_d^*$ whose concatenation has length $\ell.$ On the other hand, there are $d^{\ell}$ words in $\Sigma_d^*$ of length $\ell,$ each of which admits $\ell + 1$ bisections.

\subsection{Generalization of Andrews' Heuristic.} Naturally, we would like to find a closed-form expression for $w_{(n, \xi)}.$ Unfortunately, the sum in (\ref{enumeration}) and the integral in (\ref{integralrep}) look wholly intractable. The summation is over multi-indices which encumbers customary computational procedures such as the snake oil method and Gosper's algorithm. At the time of writing, $w_{(n, \mathbf{0}_d)}$ alone has defied evaluation for slightly over three decades. Richards and Cambanis \cite{CR} proposed the problem of calculating what they call $S(n, k) = \sum[\frac{k!}{k_1!k_2!\cdots k_n!}]^2,$ where the sum is over all nonnegative integers $k_1, \ldots, k_n,$ such that $k_1 + \cdots + k_n = k.$ Ciaob\u{a} \cite{Ciaoba} once commented that ``obtaining a closed formula\ldots seems to be an interesting and difficult combinatorial problem in itself.'' Andrews \cite{AKR} remarked that the presence of large prime factors in $S(n, k)$ makes obtaining a closed form finite product representation of $S(n, k)$ implausible. As we shall see in Theorem 4.2, Lemma 4.3, and Theorem 4.4, these same heuristic barriers persist in enumerating the $n^{\text{th}}$ order word pairs offset by $\xi.$ We can, however, glean a noteworthy recursive trait of the $w_{(n,\xi)}.$

\begin{theorem} \label{recurrence}
The number of $n^{\text{th}}$ order word pairs offset by $\xi \in \Z^d$ satisfies the recurrence
\begin{align}
w_{(n, \xi)} &= \sum_{j = 0}^{n} \binom{n + |\xi^+|}{j + \xi_{s_1}^+ + \cdots + \xi_{s_t}^+}\binom{n + |\xi^-|}{j + \xi_{s_1}^- + \cdots + \xi_{s_t}^-}
\\ & \hspace{1.5em} \times w_{(j, (\xi_{s_1}, \ldots, \xi_{s_t}))}w_{(n - j, (\xi_1, \ldots, \widehat{\xi_{s_1}}, \ldots, \widehat{\xi_{s_t}}, \ldots, \xi_d))}, \hspace{1em} \text{for}\ d \geq 2 \nonumber
\end{align}
\noindent where $t \in [1 : d - 1],$ $s_1 < \cdots < s_t$ are $t$ natural numbers selected from $\Sigma_d,$ and $\widehat{\xi_{s_1}}, \ldots, \widehat{\xi_{s_t}}$ means that the indices $\xi_{s_1}, \ldots, \xi_{s_t}$ are deleted from $\{\xi_1, \ldots, \xi_d\}.$
\end{theorem}

\begin{proof}
Partition $\Sigma_d$ into two subsets $S$ and $T$ of $t$ and $d - t$ characters, respectively. Let $s_1 < \cdots < s_t$ be the elements of $S.$ We now count the number of $n^{\text{th}}$ order word pairs $(w,w')$ offset by $\xi$ by conditioning on the value of $\rho_{s_1}(w) + \cdots + \rho_{s_t}(w).$ Recall that $w$ must have at least $\xi_{s_k}^+$ occurrences of $s_k$ for each $s_k \in S$ and at least $\xi_m^+$ occurrences of $m$ for each $m \in T.$ So $\rho_{s_1}(w) + \cdots + \rho_{s_t}(w)$ can range from $\xi_{s_1}^+ + \cdots + \xi_{s_t}^+$ to $n + |\xi^+| - \sum_{m \in T} \xi_m^+ = n + \xi_{s_1}^+ + \cdots + \xi_{s_t}^+,$ inclusive.

Now for a given $j \in [1 : n],$ assume that $j + \xi_{s_1}^+ + \cdots + \xi_{s_t}^+$ of the characters in $w$ inhabit $S.$ To preserve the assigned offsetness, $j + \xi_{s_1}^- + \cdots + \xi_{s_t}^-$ of the characters in $w'$ must inhabit $S$ as well. Once the placements in $ww'$ for characters from $S$ have been selected, they can be filled in $w_{(j, (\xi_{s_1}, \ldots, \xi_{s_t}))}$ ways. The remaining spots,which are comprised of characters from $T,$ can be filled in $w_{(n - j, (\xi_1, \ldots, \widehat{\xi_{s_1}}, \ldots, \widehat{\xi_{s_t}}, \ldots, \xi_d))}$ ways. Summing $\binom{n + |\xi^+|}{j + \xi_{s_1}^+ + \cdots + \xi_{s_t}^+}\binom{n + |\xi^-|}{j + \xi_{s_1}^- + \cdots + \xi_{s_t}^-}w_{(j, (\xi_{s_1}, \ldots, \xi_{s_t}))}w_{(n - j, (\xi_1, \ldots, \widehat{\xi_{s_1}}, \ldots, \widehat{\xi_{s_t}}, \ldots, \xi_d))}$ over all possible $j$ completes the argument.
\end{proof}

We now prove a divisibility property of $w_{(n, \xi)}$ that generalizes a fact about $w_{(n, \mathbf{0}_d)}$ conjectured by Andrews and later proven by Kolitsch in \cite{AKR}. 

\begin{lemma} \label{constantvector}
\[w_{(n, m\mathbf{1}_d)} \equiv 0\ (\textup{mod}\ d), \hspace{2.5em} \text{for all integers}\ n\ \text{and}\ m\ \text{not both}\ 0.\]
\end{lemma}

\begin{proof}
Let the cyclic group $\Z_d$ act on the set $C$ of weak compositions of $n$ into $d$ parts by cyclically permuting the parts of a composition. Pick a representative $\pi^{\mathcal{O}}$ from each orbit $\mathcal{O} \in C/\Z_d.$ Then
\begin{align}
w_{(n, m\mathbf{1}_d)} &= \sum_{\mathcal{O} \in C/\Z_d} |\mathcal{O}|\binom{|\pi^{\mathcal{O}} + (m\mathbf{1}_d)^+|}{\pi^{\mathcal{O}} + (m\mathbf{1}_d)^+}\binom{|\pi^{\mathcal{O}} + (m\mathbf{1}_d)^-|}{\pi^{\mathcal{O}} + (m\mathbf{1}_d)^-}
\\ &= \sum_{\mathcal{O} \in C/\Z_d} |\mathcal{O}|\binom{|\pi^{\mathcal{O}}|}{\pi^{\mathcal{O}}}\binom{|\pi^{\mathcal{O}} + |m|\mathbf{1}_d|}{\pi^{\mathcal{O}} + |m|\mathbf{1}_d} \nonumber
\end{align}
\noindent since one of $m^+$ and $m^-$ is 0 while the other is $|m|.$

Now let $\mathcal{O} \in C/\Z_d$ be arbitrary. By the orbit-stabilizer theorem, together with Lagrange's theorem, the sizes of the orbits divide $d,$ and so
\[\pi^{\mathcal{O}} = (\underbrace{\pi^{\mathcal{O}}_1, \ldots, \pi^{\mathcal{O}}_{|\mathcal{O}|}}_{\text{written}\ d/|\mathcal{O}|\ \text{times}}).\]
\noindent Since $\pi^{\mathcal{O}}_1 + |m|, \ldots, \pi^{\mathcal{O}}_{|\mathcal{O}|} + |m|$ cannot all be 0, their greatest common divisor is nonzero. As a result,
\begin{align}
\binom{|\pi^{\mathcal{O}} + m\mathbf{1}_d|}{\pi^{\mathcal{O}} + m\mathbf{1}_d} &= \binom{(d/|\mathcal{O}|)(\pi^{\mathcal{O}}_1 + \cdots + \pi^{\mathcal{O}}_{|\mathcal{O}|} + |\mathcal{O}||m|)}{\underbrace{\pi^{\mathcal{O}}_1 + |m|, \ldots, \pi^{\mathcal{O}}_{|\mathcal{O}|} + |m|}_{\text{written}\ d/|\mathcal{O}|\ \text{times}}}
\\ &\equiv 0\ \left(\text{mod}\ \frac{(d/|\mathcal{O}|)(\pi^{\mathcal{O}}_1 + \cdots + \pi^{\mathcal{O}}_{|\mathcal{O}|} + |\mathcal{O}||m|)}{\gcd(\pi^{\mathcal{O}}_1 + |m|, \ldots, \pi^{\mathcal{O}}_{|\mathcal{O}|} + |m|)}\right), \nonumber
\end{align}
\noindent in which we have applied Theorem 1 in \cite{Gould}. Yet $\gcd(\pi^{\mathcal{O}}_1 + |m|, \ldots, \pi^{\mathcal{O}}_{|\mathcal{O}|} + |m|)$ divides $\pi^{\mathcal{O}}_1 + \cdots + \pi^{\mathcal{O}}_{|\mathcal{O}|} + |\mathcal{O}||m|$ so that
\begin{equation}
\frac{(d/|\mathcal{O}|)(\pi^{\mathcal{O}}_1 + \cdots + \pi^{\mathcal{O}}_{|\mathcal{O}} + |\mathcal{O}||m|)}{\gcd(\pi^{\mathcal{O}}_1 + |m|, \cdots, \pi^{\mathcal{O}}_{|\mathcal{O}|} + |m|)} \equiv 0\ (\text{mod}\ d/|\mathcal{O}|).
\end{equation}
\noindent Therefore
\begin{equation}
\binom{|\pi^{\mathcal{O}} + m\mathbf{1}_d|}{\pi^{\mathcal{O}} + m\mathbf{1}_d} \equiv 0\ (\text{mod}\ d/|\mathcal{O}|),
\end{equation}
\noindent and so
\begin{equation}
|\mathcal{O}|\binom{|\pi^{\mathcal{O}}|}{\pi^{\mathcal{O}}}\binom{|\pi^{\mathcal{O}} + m\mathbf{1}_d|}{\pi^{\mathcal{O}} + m\mathbf{1}_d} \equiv 0\ (\text{mod}\ d),
\end{equation}
\noindent which shows that $w_{(n, m\mathbf{1}_d)} \equiv 0\ (\text{mod}\ d).$
\end{proof}

\begin{theorem}
Let $d \geq 2$ and $\xi \in \Z^d - \{\mathbf{0}\},$ and for each $a \in \Z,$ let $o_a(\xi)$ be the number of occurrences of $a$ in $\xi.$ Then
\begin{equation}
w_{(n, \xi)} \equiv 0\ \left(\textup{mod}\ \textup{lcm}(1, 2, \ldots, \max_{a \neq 0} o_a(\xi))\right).
\end{equation}
\end{theorem}

\begin{proof}
Note that $\max_{a \neq 0} o_a(\xi) \geq 1.$ By letting $a_0$ be a nonzero integer with the maximal number of occurrences in $\xi$, we can apply Theorem \ref{recurrence} to get
\begin{equation}
w_{(n, \xi)} = \sum_{j = 0}^{n} \binom{n + |\xi^+|}{j + ta_0^+}\binom{n + |\xi^-|}{j + ta_0^-}w_{(j, a_0\mathbf{1}_t)}w_{(n - j, (\xi_1, \ldots, \widehat{a_0\mathbf{1}_t}, \ldots, \xi_d))}
\end{equation}
\noindent for each $t \in [1 : \max_{a \neq 0} o_a(\xi)].$ It follows from Lemma \ref{constantvector} that every summand is divisible by $t.$
\end{proof}

\section{Asymptotics}
In this chapter, we extract the asymptotic behavior of $w_{(n, \xi)},$ first as $n \to \infty$ with $\xi$ fixed, then as $\|\xi\|_1 \to \infty$ in a fixed direction in $\Z^d$ with $n$ fixed, and then as the dimension $d \to \infty$ with $n$ fixed and the components of $\xi$ fixed and all the same.

\subsection{Coefficient Asymptotics of $W_{\xi}$}

Below is our main asymptotic result.

\begin{theorem}
If $\xi \in \Z^d,$ then
\begin{equation} \label{coefficientasymptotic}
w_{(n, \xi)} \sim d^{2n + d/2 + \|\xi\|_1}(4\pi n)^{(1 - d)/2}
\end{equation}
\noindent as $n \to \infty.$
\end{theorem}

This agrees with the already established leading-term asymptotic of $w_{(n, \mathbf{0}_d)}$ ascertained by Richmond and Rousseau in \cite{RR}. We shall derive the coefficient asymptotics of $W_{\xi}$ by applying Laplace's method on (\ref{enumeration}). The following version of Laplace's method for sums over lattice point translations, which is explained in \cite{GJR}, is effective for approximating sums over multi-indices.

\begin{theorem}[Greenhill, Janson, Ruci\'{n}ski] \label{gjr}
Suppose the following:

\begin{enumerate}[(i)]
\item $\mathcal{L} \subset \R^N$ is a lattice with rank $r \leq N.$\\

\item $V \subset \R^N$ is the $r$-dimensional subspace spanned by $\mathcal{L}.$\\

\item $W = V + w$ is an affine subspace parallel to $V$ for some $w \in \R^N.$\\

\item $K \subset \R^N$ is a compact convex set with nonempty interior $K^{\circ}.$\\

\item $\phi : K \rightarrow \R$ is a continuous function and the restriction of $\phi$ to $K \cap W$ has a unique maximum at some point $x_0 \in K^{\circ} \cap W.$\\

\item $\phi$ is twice continuously differentiable in a neighborhood of $x_0$ and $H_{\phi}(x_0)$ is its Hessian at $x_0.$\\

\item $\psi : K_1 \rightarrow \R$ is a continuous function on some neighborhood $K_1 \subset K$ of $x_0$ with $\psi(x_0) > 0.$\\

\item For each positive integer $n$ there is a vector $\ell_n \in \R^N$ with $\ell_n/n \in W.$\\

\item For each positive integer $n$ there is a positive real number $b_n$ and a function $a_n : (\mathcal{L} + \ell_n) \cap nK \rightarrow \R$ such that, as $n \to \infty,$
\[a_n(\ell) = O\!\left(b_ne^{n\phi(\ell/n) + o(n)}\right), \hspace{2.5em} \ell \in (\mathcal{L} + \ell_n) \cap nK,\]
\noindent and
\[a_n(\ell) = b_n\!\left(\psi(\ell/n) + o(1)\right)e^{n\phi(\ell/n)}, \hspace{2.5em} \ell \in (\mathcal{L} + \ell_n) \cap nK_1,\]
\noindent uniformly for $\ell$ in the indicated sets.\\
\end{enumerate}
\noindent Then, regarding $-H_{\phi}$ as a bilinear form on $V$ and provided $\left.\det\!\left(-H_{\phi}\right|_{V}\right) \neq 0,$ as $n \to \infty,$
\[\sum_{\ell \in (\mathcal{L} + \ell_n) \cap nK} a_n(\ell) \sim \frac{(2\pi)^{r/2}\psi(x_0)}{\det(\mathcal{L})\sqrt{\left.\det\left(-H_{\phi}\right|_{V}\!(x_0)\right)}}b_nn^{r/2}e^{n\phi(x_0)},\]
\noindent where $\det(\mathcal{L})$ is the square root of the discriminant of $\mathcal{L}.$
\end{theorem}

\begin{proof}[Proof of Theorem 5.1]
The asymptotic trivially holds for $d = 1$ since $w_{(n, \xi)} = 1$ for all $\xi \in \Z,$ so we let $d \geq 2.$ To start, write
\[(n_1, n_2, \ldots, n_d) = (n_1 - n, n_2, \ldots, n_d) + n(1, 0, \ldots, 0)\]
\noindent so that if $(n_1, \ldots, n_d) \in \Z^d$ and $n_1 + \cdots + n_d = n,$ then $(n_1 - n, n_2, \ldots, n_d)$ is an integer vector whose components sum to 0. This suggests that we let $\mathcal{L}$ be the root lattice $A_{d - 1} = \{\omega \in \Z^d : \omega_1 + \cdots + \omega_d = 0\},$ let $V$ be the real linear span of $A_{d - 1},$ let $w = (1, 0, \ldots, 0),$ let $W = V + w,$ and let $\ell_n = nw.$ Letting $K$ be the unit hypercube $[0, 1]^d,$ we have
\begin{equation} \label{latticesum}
w_{(n, \xi)} = (n + |\xi^+|)!(n + |\xi^-|)!\sum_{\ell \in (A_{d - 1} + nw) \cap nK} a_n(\ell)
\end{equation}
\noindent where
\begin{equation}
a_n(\ell) = \prod_{j = 1}^{d} \frac{1}{\ell_j!(\ell_j + |\xi_j|)!}.
\end{equation}
\noindent Let $x_j = \ell_j/n$ for all $j \in [1 : d].$ Applying Stirling's approximation in the form
\begin{equation}
\log(n!) = n\log n - n + \frac{1}{2}\log(\max\{n, 1\}) + \frac{1}{2}\log2\pi + O\!\left(\frac{1}{n + 1}\right), \hspace{1em} \text{for}\ n \geq 0
\end{equation}
\noindent we obtain, uniformly for $\ell \in (A_{d - 1} + nw) \cap nK$ with $n$ sufficiently large,
\begin{align*}
\log(a_n(\ell)) &= -\sum_{j = 1}^{d} \left(\log(\ell_j!) + \log((\ell_j + |\xi_j|)!)\right)
\\ &= -\sum_{j = 1}^{d} \left((2\ell_j + |\xi_j| + 1)\log n - (2\ell_j + |\xi_j|) + \log2\pi\right)
\\ &\hspace{1.5em} -n\sum_{j = 1}^{d} \left(x_j\log x_j + (x_j + |\xi_j|/n)\log(x_j + |\xi_j|/n)\right)
\\ &\hspace{1.5em} -\frac{1}{2}\sum_{j = 1}^{d} \left(\log(\max\{x_j, 1/n\}) + \log(\max\{x_j + |\xi_j|/n, 1/n\})\right)
\\ &\hspace{1.5em} -\sum_{j = 1}^{d} \left(O\!\left(\frac{1}{\ell_j + 1}\right) + O\!\left(\frac{1}{\ell_j + |\xi_j| + 1}\right)\right)
\\ &= -(2n + \|\xi\|_1 + d)\log n + 2n + \|\xi\|_1 - d\log2\pi
\\ &\hspace{1.5em} -n\sum_{j\ \text{with}\ x_j \neq 0} \left(x_j\log x_j + (x_j + |\xi_j|/n)(\log x_j + |\xi_j|/\ell_j - O(\xi_j^2/\ell_j^2))\right)
\\ &\hspace{1.5em} -n\sum_{j\ \text{with}\ x_j = 0} \left(x_j\log x_j + (|\xi_j|/n)\log(|\xi_j|/n)\right) - \sum_{j\ \text{with}\ x_j = \xi_j = 0} \log(1/n)
\\ &\hspace{1.5em} -\frac{1}{2}\sum_{j,\, x_j = 0, \xi_j \neq 0} \left(2\log(1/n) + \log|\xi_j|\right) - \frac{1}{2}\sum_{j,\, x_j \neq 0} \left(2\log x_j + O(|\xi_j|/\ell_j)\right)
\\ &\hspace{1.5em} -\sum_{j = 1}^{d} \left(O\!\left(\frac{1}{\ell_j + 1}\right) + O\!\left(\frac{1}{\ell_j + |\xi_j| + 1}\right)\right)
\\ &= -(2n + \|\xi\|_1 + d)\log n + 2n - d\log2\pi
\\ &\hspace{1.5em} +\sum_{j,\, x_j = 0} \left(|\xi_j| - |\xi_j|\log|\xi_j| - \frac{1}{2}(\log|\xi_j|)\mathbbold{1}_{\xi_j \neq 0}\right) - n\sum_{j = 1}^{d} 2x_j\log x_j
\\ &\hspace{1.5em} -\sum_{j = 1}^{d} \left(\log(\max\{x_j, 1/n\}) + |\xi_j|\log(\max\{x_j, 1/n\})\right)
\\ &\hspace{1.5em} -\sum_{j = 1}^{d}\! \left(\!O\!\left(\frac{1}{\ell_j + 1}\right)\! +\! O\!\left(\frac{1}{\ell_j + |\xi_j| + 1}\right)\!\right)\! +\! \sum_{j,\, x_j \neq 0}\! \left(\!O\!\left(\frac{1}{\ell_j}\right)\! +\! O\!\left(\frac{1}{\ell_j^2}\right)\!\right).
\end{align*}
\noindent where $\mathbbold{1}_{\xi_j \neq 0} = 1$ if $\xi_j \neq 0$ and 0 otherwise. We can therefore write
\[a_n(\ell) = b_n\psi(\ell/n)e^{n\phi(\ell/n) + \sum\limits_{j,\, x_j = 0} \left(|\xi_j| - \left(|\xi_j| + \frac{1}{2}\mathbbold{1}_{\xi_j \neq 0}\right)\log|\xi_j|\right)}\!\left(1 + O\!\left(\frac{1}{\min_j \ell_j + 1}\right)\right),\]
\noindent where, for $x \in \R^d,$
\[b_n = \frac{e^{2n}}{(2\pi)^dn^{2n + \|\xi\|_1 + d}}, \hspace{2.5em} \psi(x) = \prod_{j = 1}^{d} \frac{1}{x_j^{|\xi_j| + 1}}, \hspace{2.5em} \phi(x) = -\sum_{j = 1}^{d} 2x_j\log x_j,\]
\noindent unless if some $x_j$ is 0, in which case we replace it with $1/n$ in the formula for $\psi.$ This means that
\[a_n(\ell) = O\!\left(b_ne^{n\phi(\ell/n) + O(1)}\prod_{j = 1}^{d} \frac{1}{(1/n)^{|\xi_j| + 1}}\right) = O\!\left(b_ne^{n\phi(\ell/n) + O(\log n)}\right)\]
\noindent for $\ell \in (A_{d - 1} + nw) \cap nK.$ On $K^{\circ},$
\[\sum_{j,\, x_j = 0} \left(|\xi_j| - \left(|\xi_j| + \frac{1}{2}\mathbf{1}_{\xi_j \neq 0}\right)\log|\xi_j|\right)\]
\noindent is an empty sum. Lastly, $\psi$ is continuous and positive on $K^{\circ}.$ All that remains is to prove that $\phi|_{K \cap W}$ has a unique maximum at some point located inside $K^{\circ} \cap W$ and that there is a neighborhood $K_1 \subset K$ of that point that will satisfy conditions (vii) and (ix) of Theorem \ref{gjr}.

By Jensen's inequality,
\begin{equation}
\phi(x) = 2\sum_{j = 1}^{d} x_j\log\!\left(\frac{1}{x_j}\right) \leq 2\log d,
\end{equation}
\noindent for all $x \in K \cap W = \{x \in K : x_1 + \cdots + x_d = 1\}.$ Since equality holds if and only if $x_1 = x_2 = \cdots = x_d,$ we see that $\phi|_{K \cap W}$ attains its maximum value of $2\log d$ at the point $x_0 = (1/d)\mathbf{1}_d$ only. Finally, the Hessian $H_{\phi}(x)$ is diagonal with entries $-2/x_j.$ Hence $H_{\phi}(x_0) = -2dI_d.$

Now set $K_1 = \left(\frac{1}{2d}, \frac{3}{2d}\right)^d.$ As a continuous function, $\psi$ is bounded on the compact set $\left[\frac{1}{2d}, \frac{3}{2d}\right]^d.$ Furthermore, if $\ell \in (A_{d - 1} + nw) \cap nK_1,$ then
\[0 \leq \frac{1}{\min_j \ell_j + 1} \leq \frac{1}{\frac{n}{2d} + 1} \to 0\]
\noindent as $n \to \infty.$ Hence $a_n(\ell) = b_n(\psi(\ell/n) + o(1))e^{n\phi(\ell/n)}$ for $\ell \in (A_{d - 1} + nw) \cap nK_1.$ Every assumption of Theorem \ref{gjr} has now been verified.

It is well known that $A_{d - 1}$ has rank $d - 1$ and discriminant $d,$ so we are left to calculate $\left.\det\left(-H_{\phi}\right|_V\!(x_0)\right).$ Let $\{\mathbf{e}_k\}_{k = 1}^{d - 1}$ be a basis for $V.$ Then
\begin{equation}
\left.\det\left(-H_{\phi}\right|_V\!(x_0)\right) = \frac{\det([\mathbf{e}_i^{\mathsf{T}}(2dI_d)\mathbf{e}_j]_{i, j = 1}^{d - 1})}{\det([\mathbf{e}_i^{\mathsf{T}}\mathbf{e}_j]_{i, j = 1}^{d - 1})} = (2d)^{d - 1}.
\end{equation}
\noindent Therefore
\begin{align} \label{laplace}
\sum_{\ell \in (A_{d - 1} + nw) \cap nK} a_n(\ell) &\sim \frac{(2\pi)^{(d - 1)/2}d^{d + \|\xi\|_1}}{d^{1/2}(2d)^{(d - 1)/2}}\frac{e^{2n}}{(2\pi)^dn^{2n + \|\xi\|_1 + d}}n^{(d - 1)/2}e^{2n\log d}
\\ &= \frac{d^{2n + d/2 + \|\xi\|_1}e^{2n}}{2^d\pi^{(d + 1)/2}n^{2n + \|\xi\|_1 + (d + 1)/2}}. \nonumber
\end{align}
\noindent Using (\ref{laplace}) and Stirling's formula on (\ref{latticesum}) yields
\begin{align} \label{asymptotic}
w_{(n, \xi)} &\sim \sqrt{2\pi(n + |\xi^+|)}\left(\frac{n + |\xi^+|}{e}\right)^{n + |\xi^+|}\sqrt{2\pi(n + |\xi^-|)}\left(\frac{n + |\xi^-|}{e}\right)^{n + |\xi^-|}
\\ &\hspace{2.5em} \times \frac{d^{2n + d/2 + ||\xi||_1}e^{2n}}{2^d\pi^{(d + 1)/2}n^{2n + ||\xi||_1 + (d + 1)/2}}. \nonumber
\end{align}
\noindent Simplifying (\ref{asymptotic}) via the asymptotic expressions $\sqrt{n + |\xi^+|} \sim \sqrt{n + |\xi^-|} \sim \sqrt{n},$ $(n + |\xi^+|)^{n + |\xi^+|} \sim n^{n + |\xi^+|}e^{|\xi^+|},$ and $(n + |\xi^-|)^{n + |\xi^-|} \sim n^{n + |\xi^-|}e^{|\xi^-|}$ (each as $n \to \infty$) begets (\ref{coefficientasymptotic}).
\end{proof}

\subsection{Stationary Phase Approximation of $w_{(n, \xi)}$}

We now derive the following leading order estimate of $w_{(n, \xi)}$ as $\|\xi\|_1 \to \infty$ in a fixed direction in $\Z^d$ with $n$ fixed.

\begin{theorem}
If $\xi \in \Z^d - \{\mathbf{0}_d\},$ then
\begin{equation}
w_{(n, \lambda\xi)} = \frac{(2\pi)^{1 - 3d/2}}{\sqrt{\|\xi\|_1(d + 1)}}d^{\lambda\|\xi\|_1 + 2n + d/2 + 1}\lambda^{-d/2}(1 + o(1))
\end{equation}
\noindent as $\lambda \to \infty$ with $n$ fixed.
\end{theorem}

To prove this, we will employ the integral representation
\[w_{(n, \lambda\xi)} = \frac{1}{(2\pi)^d}\int\limits_{[0, 2\pi]^d} \left[e^{-i\lambda\xi^{\textsf{T}}\theta}\left(\sum_{k = 1}^{d} e^{i \theta_k}\right)^{\lambda|\xi^+|}\!\left(\sum_{k = 1}^{d} e^{-i \theta_k}\right)^{\lambda|\xi^-|}\!\left|\sum_{k = 1}^{d} e^{i \theta_k}\right|^{2n}\right]\mathrm{d}\theta.\]
\noindent and allow the positive parameter $\lambda$ to tend to $+\infty.$ This oscillatory integral is mostly amenable to the following variation of stationary phase method \cite[Theorem 4.1]{PW}.

\begin{theorem}[Pemantle, Wilson] \label{stationaryphase}
Let $A$ and $\varphi$ be complex-valued analytic functions on a compact neighborhood $\mathcal{N}$ of the origin in $\R^d$ and suppose that the real part of $\varphi$ is nonnegative, vanishing only at the origin. Suppose that the Hessian matrix $H_{\varphi}$ of $\varphi$ at the origin is nonsingular. Denoting $\mathcal{I}(\lambda) := \int_{\mathcal{N}} A(\mathbf{x})e^{-\lambda\varphi(\mathbf{x})}\,\mathrm{d}\mathbf{x},$ there is an asymptotic series
\[\mathcal{I}(\lambda) \sim \sum_{\ell \geq 0} c_{\ell}\lambda^{-d/2 - \ell}\]
\noindent where
\[c_0 = A(\mathbf{0}_d)\frac{(2\pi)^{-d/2}}{\sqrt{\det(H_{\varphi}(\mathbf{0}_d))}}\]
\noindent and the choice of sign of the square root is defined by taking the product of the principal square roots of the eigenvalues of $H_{\varphi}(\mathbf{0}_d).$
\end{theorem}

We say ``mostly'' because we actually need the following slight modification of Theorem \ref{stationaryphase}.

\begin{corollary} \label{continuousamplitude}
Let $A$ be a complex-valued continuous function on a compact neighborhood $\mathcal{N}$ of the origin in $\R^d.$ Let $\varphi$ be a complex-valued analytic function satisfying the conditions of Theorem \ref{stationaryphase}. Denoting $\mathcal{I}(\lambda) := \int_{\mathcal{N}} A(\mathbf{x})e^{-\lambda\varphi(\mathbf{x})}\,\mathrm{d}\mathbf{x},$ then
\[\mathcal{I}(\lambda) = A(\mathbf{0}_d)\frac{(2\pi\lambda)^{-d/2}}{\sqrt{\det(H_{\varphi}(\mathbf{0}_d))}}(1 + o(1))\]
\noindent as $\lambda \to \infty.$ The choice of sign of the square root is defined by taking the product of the principal square roots of the eigenvalues of $H_{\varphi}(\mathbf{0}_d).$
\end{corollary}

\begin{proof}
By the Stone-Weierstrass Theorem, there exists a sequence $\{P_n\}_{n = 1}^{\infty}$ of polynomials in $\C[x_1, \ldots, x_d]$ that converges uniformly to $A$ over $\mathcal{N}.$ (Technically, it is the complex unital $^*$-algebra generated by $\C[x_1, \ldots, x_d]$ that is dense in $C(\mathcal{N}, \C),$ the algebra of complex-valued continuous functions on $\mathcal{N},$ endowed with the infinity norm. This $^*$-algebra, however, is $\C[x_1, \ldots, x_d]$ itself since the variables $x_1, \ldots, x_d$ are real, and so the set of complex-coefficient polynomials on $\mathcal{N}$ is closed under complex conjugation.) Now for each $n \in \N,$ define
\[\mathcal{I}_n(\lambda) := \int\limits_{\mathcal{N}} P_n(\mathbf{x})e^{-\lambda\varphi(\mathbf{x})}\,\mathrm{d}\mathbf{x}.\]
\noindent Then for every $\lambda \in [0, \infty),$
\[\left|\mathcal{I}_n(\lambda) - \mathcal{I}(\lambda)\right| \leq \int\limits_{\mathcal{N}} \left|(P_n(\mathbf{x}) - A(\mathbf{x}))e^{-\lambda\varphi(\mathbf{x})}\right|\mathrm{d}\mathbf{x} \leq \sup_{\mathbf{x} \in \mathcal{N}} |P_n(\mathbf{x}) - A(\mathbf{x})|\int\limits_{\mathcal{N}}\mathrm{d}\mathbf{x} \to 0\]
\noindent as $n \to \infty.$ Therefore $\{\mathcal{I}_n\}_{n = 1}^{\infty}$ converges uniformly to $\mathcal{I}.$ Since each $P_n$ is analytic, Theorem \ref{stationaryphase} tells us that
\[\mathcal{I}_n(\lambda) = P_n(\mathbf{0}_d)\frac{(2\pi\lambda)^{-d/2}}{\sqrt{\det(H_{\varphi}(\mathbf{0}_d))}}(1 + o(1))\]
\noindent as $\lambda \to \infty.$ If $A(\mathbf{0}_d) \neq 0,$ then we can use the Moore-Osgood Theorem \cite[Theorem 7.11]{Rudin} to conclude that
\[\lim_{\lambda \to \infty} \frac{\mathcal{I}(\lambda)}{\frac{A(\mathbf{0}_d)(2\pi\lambda)^{-d/2}}{\sqrt{\det(H_{\varphi}(\mathbf{0}_d))}}} = \lim_{\lambda \to \infty} \lim_{n \to \infty} \frac{\mathcal{I}_n(\lambda)}{\frac{P_n(\mathbf{0}_d)(2\pi\lambda)^{-d/2}}{\sqrt{\det(H_{\varphi}(\mathbf{0}_d))}}} = \lim_{n \to \infty} \lim_{\lambda \to \infty} \frac{\mathcal{I}_n(\lambda)}{\frac{P_n(\mathbf{0}_d)(2\pi\lambda)^{-d/2}}{\sqrt{\det(H_{\varphi}(\mathbf{0}_d))}}} = 1.\]
\noindent If $A(\mathbf{0}_d) = 0,$ then the cruder limit
\[\lim_{\lambda \to \infty} \mathcal{I}(\lambda) = \lim_{\lambda \to \infty} \lim_{n \to \infty} \mathcal{I}_n(\lambda) = 0\]
\noindent is adequate.
\end{proof}

\begin{proof}[Proof of Theorem 5.3]
First we invoke the change of variables given by the map $\mathbf{f}_d : \R^d \rightarrow \R^d$ with components
\[\theta_j = -\tau_{j - 1} + \tau_j + \delta, \hspace{1em} \text{for}\ 1 \leq j \leq d\]
\noindent where we have introduced the placeholders $\tau_0 \equiv \tau_d \equiv 0$ in order to avoid messy casework later on. (One may envision this as a transformation from the Cartesian coordinate system to a sort-of ``signed distance from $A_{d - 1}$'' coordinate system where, given a point $\theta_0 \in \R^d,$ $-\delta$ gives the value of $t$ for which the hyperplane $\theta_1 + \cdots + \theta_d = 0$ intersects with the normal line $\ell(t) = \theta_0 + t\mathbf{1}_d$ and $(\tau_1, \ldots, \tau_{d - 1})$ are the coordinates of this intersection point relative to the standard basis of $A_{d - 1}.$) The Jacobian matrix of $\mathbf{f}_d$ is
\[\mathbf{J}_{\mathbf{f}_d}(\tau_1, \ldots, \tau_{d - 1}, \delta) = \left[\begin{matrix} \frac{\partial\theta_1}{\partial\tau_1} & \frac{\partial\theta_1}{\partial\tau_2} & \cdots & \frac{\partial\theta_1}{\partial\tau_{d - 1}} & \frac{\partial\theta_1}{\partial\delta} \\ \frac{\partial\theta_2}{\partial\tau_1} & \frac{\partial\theta_2}{\partial\tau_2} & \cdots & \frac{\partial\theta_2}{\partial\tau_{d - 1}} & \frac{\partial\theta_2}{\partial\delta} \\ \frac{\partial\theta_3}{\partial\tau_1} & \frac{\partial\theta_3}{\partial\tau_2} & \cdots & \frac{\partial\theta_3}{\partial\tau_{d - 1}} & \frac{\partial\theta_3}{\partial\delta} \\ \vdots & \vdots & \ddots & \vdots & \vdots \\ \frac{\partial\theta_d}{\partial\tau_1} & \frac{\partial\theta_d}{\partial\tau_2} & \cdots & \frac{\partial\theta_d}{\partial\tau_{d - 1}} & \frac{\partial\theta_d}{\partial\delta} \end{matrix}\right] = \left[\begin{matrix} 1 & 0 & \cdots & 0 & 1 \\ -1 & 1 & \ddots & 0 & 1 \\ 0 & -1 & \ddots & 0 & 1 \\ \vdots & \vdots & \ddots & \ddots & \vdots \\ 0 & 0 & \cdots & -1 & 1 \end{matrix}\right],\]
\noindent a $d \times d$ matrix with 1s on the main diagonal, 1s in the last column, $-1$s on the subdiagonal, and 0s elsewhere.

We use induction to show that the Jacobian determinant is equal to $d.$ For the base case $d = 1,$ $\mathbf{f}_1(\theta_1) = \delta$ which trivially has a Jacobian determinant of $1.$ Now assume the induction hypothesis holds for some integer $k,$ that is, $\det(\mathbf{J}_{\mathbf{f}_k}) = k.$ Laplace expansion along the first row of $\mathbf{J}_{\mathbf{f}_{k + 1}}$ yields
\[\det(\mathbf{J}_{\mathbf{f}_{k + 1}}) = \det(\mathbf{J}_{\mathbf{f}_k}) + (-1)^{k + 2}\det(J_{-1, k})\]
\noindent where $J_{-1, k}$ is the $k \times k$ Jordan block with eigenvalue $-1.$ We thus have
\[\det(\mathbf{J}_{\mathbf{f}_{k + 1}}) \underbrace{=}_{\text{ind. hyp.}} k + (-1)^{k + 2}(-1)^k = k + 1.\]

Since the Jacobian determinant is a nonzero constant, $\mathbf{f}_d$ is invertible and $\mathbf{f}_d^{-1}$ is differentiable. The transformation produces the following integral representation:
\begin{align}
w_{(n, \lambda\xi)} &= \frac{d}{(2\pi)^d}\!\int\limits_{\mathbf{f}_d^{-1}([-\pi, \pi]^d)}\!\left[e^{-i\lambda\xi^{\mathsf{T}}\mathbf{f}_d^{-1}(\theta)}\left(\sum_{k = 1}^{d} e^{i(\tau_k - \tau_{k - 1} + \delta)}\right)^{\lambda|\xi^+|}\!\left(\sum_{k = 1}^{d} e^{-i(\tau_k - \tau_{k - 1} + \delta)}\right)^{\lambda|\xi^-|}\right.
\\ & \hspace{10em} \left.\times\left|\sum_{k = 1}^{d} e^{i(\tau_k - \tau_{k - 1} + \delta)}\right|^{2n}\right]\!\mathrm{d}\tau_1\wedge\cdots\wedge\mathrm{d}\tau_{d - 1}\wedge\mathrm{d}\delta. \nonumber
\end{align}
\noindent Note $e^{i\delta}$ and $e^{-i\delta}$ can be factored out of $\left(\sum e^{i(\tau_k - \tau_{k - 1} + \delta)}\right)^{\lambda|\xi^+|}$ and $\left(\sum e^{-i(\tau_k - \tau_{k - 1} + \delta)}\right)^{\lambda|\xi^-|},$ respectively, to get
\begin{align}
w_{(n, \lambda\xi)} &= \frac{d}{(2\pi)^d}\!\int\limits_{\mathbf{f}_d^{-1}([-\pi, \pi]^d)}\! \left[e^{-i\lambda\xi^{\mathsf{T}}\mathbf{f}_d^{-1}(\theta)}e^{i\delta\lambda|\xi^+|}\!\left(\sum_{k = 1}^{d} e^{i(\tau_k - \tau_{k - 1})}\right)^{\lambda|\xi^+|}\!e^{-i\delta\lambda|\xi^-|}\left(\sum_{k = 1}^{d} e^{-i(\tau_k - \tau_{k - 1})}\right)^{\lambda|\xi^-|}\right. \nonumber
\\ & \hspace{10em} \left.\times \left|\sum_{k = 1}^{d} e^{i(\tau_k - \tau_{k - 1})}\right|^{2n}\right]\!\mathrm{d}\tau_1\wedge\cdots\wedge\mathrm{d}\tau_{d - 1}\wedge\mathrm{d}\delta. \nonumber
\end{align}
\noindent Yet
\begin{align*}
e^{-i\lambda\xi^{\mathsf{T}}\mathbf{f}_d^{-1}(\theta)} = e^{-i\lambda\sum\limits_{k = 1}^{d}\xi_k(\tau_k - \tau_{k - 1} + \delta)} = e^{-i\lambda\sum\limits_{k = 1}^{d}\xi_k(\tau_k - \tau_{k - 1})}\cdot e^{-i\lambda\delta\sum\limits_{k = 1}^{d}\xi_k}
\end{align*}
\noindent and
\[e^{i\delta\lambda|\xi^+|}e^{-i\delta\lambda|\xi^-|} = e^{i\lambda\delta\sum\limits_{k = 1}^{d}\xi_k}\]
\noindent so that every instance of $\delta$ in the integrand cancels out, thus yielding
\begin{align*}
w_{(n, \lambda\xi)} &= \frac{d}{(2\pi)^d}\int\limits_{\mathbf{f}_d^{-1}([-\pi, \pi]^d)} \left[e^{-i\lambda\sum\limits_{k = 1}^{d}\xi_k(\tau_k - \tau_{k - 1})}\!\left(\sum_{k = 1}^{d} e^{i(\tau_k - \tau_{k - 1})}\right)^{\lambda|\xi^+|}\right.
\\ & \hspace{10em} \left.\times \left(\sum_{k = 1}^{d} e^{-i(\tau_k - \tau_{k - 1})}\right)^{\lambda|\xi^-|}\left|\sum_{k = 1}^{d} e^{i(\tau_k - \tau_{k - 1})}\right|^{2n}\right]\!\mathrm{d}\tau_1\wedge\cdots\wedge\mathrm{d}\tau_{d - 1}\wedge\mathrm{d}\delta. \nonumber
\end{align*}
\noindent Integrate with respect to $\delta$ first to get
\begin{align} \label{cov}
w_{(n, \lambda\xi)} &= \frac{d}{(2\pi)^d} \int\limits_{\mathcal{N}} \left[\left(2\pi + m(\tau) - M(\tau)\right)e^{-i\lambda\sum\limits_{k = 1}^{d}\xi_k(\tau_k - \tau_{k - 1})}\!\left(\sum_{k = 1}^{d} e^{i(\tau_k - \tau_{k - 1})}\right)^{\lambda|\xi^+|}\right.
\\ & \hspace{10em} \left.\times \left(\sum_{k = 1}^{d} e^{-i(\tau_k - \tau_{k - 1})}\right)^{\lambda|\xi^-|}\left|\sum_{k = 1}^{d} e^{i(\tau_k - \tau_{k - 1})}\right|^{2n}\right]\!\mathrm{d}\tau \nonumber
\\ &= \frac{d}{(2\pi)^d} \int\limits_{\mathcal{N}} A(\tau)e^{-\lambda\varphi(\tau)}\,\mathrm{d}\tau,
\end{align}
\noindent where
\begin{align*}
\tau &= (\tau_1, \ldots, \tau_{d - 1}), \\
d\tau &= d\tau_1 \wedge \cdots \wedge d\tau_{d - 1}, \\
\mathcal{N} &= \mathbf{f}_d^{-1}(\text{span}_{\R}(A_{d - 1}) \cap [-\pi, \pi]^d), \\
m(\tau) &= \min\{-\tau_0 + \tau_1, \ldots, -\tau_{d - 1} + \tau_d\}, \\
M(\tau) &= \max\{-\tau_0 + \tau_1, \ldots, -\tau_{d - 1} + \tau_d\}, \\
A(\tau) &= (2\pi + m(\tau) - M(\tau))\left|\sum_{k = 1}^{d} e^{i(\tau_k - \tau_{k - 1})}\right|^{2n}, \\
\varphi(\tau) &= -|\xi^+|\log\!\left(E(\tau)\right) - |\xi^-|\log\!\left(\overline{E(\tau)}\right) + i\sum_{k = 1}^{d} \xi_k(\tau_k-\tau_{k - 1}) \\
\text{and}\ E(\tau) &= \sum_{k = 1}^{d} e^{i(\tau_k - \tau_{k - 1})},
\end{align*}
where the principal branch of the logarithm is taken.

By the triangle inequality,
\begin{align*}
\left|\sum_{k = 1}^{d} e^{i(\tau_k - \tau_{k - 1})}\right| \leq \sum_{k = 1}^{d} \left|e^{i\tau_k - \tau_{k - 1})}\right| = d,
\end{align*}
\noindent and equality occurs if and only if
\[e^{i(\tau_1-\tau_0)} = \cdots = e^{i(\tau_d-\tau_{d - 1})}.\]
Referring back to our change of variables, this means that all of the $e^{i\theta_j}$ need to equal each other in order for $E$ to attain its maximum value of $d.$ But since our region of integration in (\ref{cov}) demands that we have $\theta \in \text{span}_{\R}(A_{d - 1}) \cap [-\pi, \pi]^d,$ if any $\theta_j$ equals $\pm\pi,$ then there must be an opposing $\theta_k$ equal to $-\theta_j.$ In this case $(2\pi - m(\tau) - M(\tau)) = 0,$ and so the boundary points contribute nothing to $\mathcal{I}(\lambda).$ Hence the only stationary point of $E$ in $\mathbf{f}_d^{-1}(\text{span}_{\R}(A_{d - 1}) \cap [-\pi, \pi]^d)$ that contributes to the integral is $\mathbf{f}_d^{-1}(\mathbf{0}_d),$ in which case we have
\[\tau_1-\tau_0 = \cdots = \tau_d-\tau_{d - 1} = 0\]
\noindent so that each $\tau_j = 0.$ The same is true for
\[\sum_{k = 1}^{d} e^{-i(\tau_k - \tau_{k - 1})} = \sum_{k = 1}^{d} \overline{e^{i(\tau_k - \tau_{k - 1})}}.\]
Therefore
\begin{align*}
\Re \varphi(\tau) &= |\xi^+|\log\left|\sum_{k = 1}^{d} e^{i(\tau_k - \tau_{k - 1})}\right|^{-1} + |\xi^-|\log\left|\sum_{k = 1}^{d} e^{-i(\tau_k - \tau_{k - 1})})\right|^{-1}
\\ &\geq |\xi^+|\log d^{-1} + |\xi^-|\log d^{-1}
\\ &= -\|\xi\|_1\log d
\end{align*}
\noindent for every $\tau \in \mathcal{N}$ for which $\sum_{k = 1}^{d} e^{i(\tau_k - \tau_{k - 1})} \neq 0$ with equality occurring precisely when $\tau = \mathbf{0}_d.$ Also, since plugging in $\tau = \mathbf{0}_d$ gives $p_{1, d}(1, \ldots, 1) \neq 0,$ the phase function $\varphi$ is analytic on every compact neighborhood of $\mathbf{0}_d$ contained in $\mathcal{N}$ for which $\sum_{k = 1}^{d} e^{i(\tau_k - \tau_{k - 1})} \neq 0.$ (Two things are worth pointing out here. First, $\mathcal{N},$ as the inverse image of a compact set under a continuous map, is itself compact, and since $p_{1, d}$ is continuous, such nonempty compact neighborhoods exist. Second, the set of points for which $p_{1, d}$ equal 0 is contained in the zero set of $A(\tau)$.) Provided that $H_{\varphi}(\mathbf{0}_d)$ is invertible, we can apply Corollary \ref{continuousamplitude} to
\begin{equation}
w_{(n, \lambda\xi)} = \frac{d^{\lambda\|\xi\|_1 + 1}}{(2\pi)^d}\int\limits_{\mathcal{N}} A(\tau)e^{-\lambda(\varphi(\tau) + \|\xi\|_1\!\log d)}\,\mathrm{d}\tau.
\end{equation}

All that remains is to calculate $H_{\varphi}.$ We find that, for $1 \leq j \leq d - 1,$
\[\varphi_{\tau_j} = \frac{-|\xi^+|(ie^{i(\tau_j - \tau_{j - 1})} - ie^{i(\tau_{j + 1} - \tau_j)})}{E(\tau)} - \frac{|\xi^-|(ie^{i(\tau_j - \tau_{j + 1})} - ie^{i(\tau_{j - 1} - \tau_j)})}{\overline{E(\tau)}} + i(\xi_j - \xi_{j + 1}),\]
\noindent and so for all $j, k$ with $1 \leq j \leq k \leq d - 1$ and $k - j > 1$
\begin{align*}
\varphi_{\tau_j\tau_j} &= -|\xi^+|\frac{(-e^{i(\tau_j - \tau_{j - 1})} - e^{i(\tau_{j + 1} - \tau_j)})E(\tau) - (ie^{i(\tau_j -\tau_{j - 1})} - ie^{i(\tau_{j + 1} - \tau_j)})^2}{[E(\tau)]^2}
\\ & \hspace{1.5em} -|\xi^-|\frac{(-e^{i(\tau_j - \tau_{j + 1})} - e^{i(\tau_{j - 1} - \tau_j)})\overline{E(\tau)} - (ie^{i(\tau_j - \tau_{j + 1})} - ie^{i(\tau_{j - 1} - \tau_j)})^2}{\left[\overline{E(\tau)}\right]^2}
\\ \varphi_{\tau_{j + 1}\tau_j} &= -|\xi^+|\frac{e^{i(\tau_{j + 1} - \tau_j)}E(\tau) - (ie^{i(\tau_j -\tau_{j - 1})} - ie^{i(\tau_{j + 1} - \tau_j)})(ie^{i(\tau_{j + 1} - \tau_j)} - ie^{i(\tau_{j + 2} - \tau_{j + 1})})}{[E(\tau)]^2}
\\ & \hspace{1.5em} -|\xi^-|\frac{e^{i(\tau_j - \tau_{j + 1})}\overline{E(\tau)} - (ie^{i(\tau_j - \tau_{j + 1})} - ie^{i(\tau_{j - 1} - \tau_j)})(ie^{i(\tau_{j + 1} - \tau_{j + 2})} - ie^{i(\tau_j - \tau_{j + 1})})}{\left[\overline{E(\tau)}\right]^2}
\\ \varphi_{\tau_k\tau_j} &= |\xi^+|\frac{(ie^{i(\tau_j - \tau_{j - 1})} - ie^{i(\tau_{j + 1} - \tau_j)})(ie^{i(\tau_k - \tau_{k - 1})} - ie^{i(\tau_{k + 1} - \tau_k)})}{[E(\tau)]^2}
\\ & \hspace{1.5em} +|\xi^-|\frac{(ie^{i(\tau_j - \tau_{j + 1})} - ie^{i(\tau_{j - 1} - \tau_j)})(ie^{i(\tau_k - \tau_{k + 1})} - ie^{i(\tau_{k - 1} - \tau_k)})}{\left[\overline{E(\tau)}\right]^2}.
\end{align*}
\noindent The remaining mixed partial derivatives follow from Clairaut's theorem. Thus
\begin{align*}
\varphi_{\tau_j\tau_j}(\mathbf{0}_d) &= -|\xi^+|\frac{-2d}{d^2} - |\xi^-|\frac{-2d}{d^2} = \frac{2\|\xi\|_1}{d}
\\ \varphi_{\tau_j\tau_{j + 1}}(\mathbf{0}_d) &= \varphi_{\tau_{j + 1}\tau_j}(\mathbf{0}_d) = -|\xi^+|\frac{d}{d^2} - |\xi^-|\frac{d}{d^2} = -\frac{\|\xi\|_1}{d}
\\ \varphi_{\tau_j\tau_k}(\mathbf{0}_d) &= \varphi_{\tau_k\tau_j}(\mathbf{0}_d) = 0,
\end{align*}
\noindent and so
\[H_{\varphi}(\mathbf{0}_d) = -\frac{\|\xi\|_1}{d}\left[\begin{matrix} -2 & 1 & 0 & \cdots & 0 \\ 1 & -2 & 1 & \ddots & 0 \\ 0 & 1 & \ddots & \ddots & 0 \\ 0 & 0 & \ddots & \ddots & 1 \\ 0 & 0 & \cdots & 1 & -2 \end{matrix}\right],\]
\noindent which is a $d \times d$ symmetric tridiagonal Toeplitz matrix. Using formula (6) from \cite{HO}
\begin{align*}
\det(H_{\varphi}(\mathbf{0}_d)) &= \left(-\frac{\|\xi\|_1}{d}\right)^d\lim_{D \to -2^-} \frac{(-1)^d\sinh\!\left((d + 1)\cosh^{-1}(-D/2)\right)}{\sinh\!\left(\cosh^{-1}(-D/2)\right)}
\\ &= \left(\frac{\|\xi\|_1}{d}\right)^d(d + 1)
\end{align*}
\noindent By Proposition 2.1 of \cite{KST}, the eigenvalues of $H_{\varphi}(\mathbf{0}_d)$ are
\[-\frac{\|\xi\|_1}{d}\left(-2 - 2\cos\frac{k\pi}{d + 1}\right) = \frac{4\|\xi\|_1}{d}\cos^2\!\left(\frac{k\pi}{2(d + 1)}\right), \hspace{1em} \text{for}\ k = 1, 2, \ldots, d,\]
\noindent and so the product of the principal square roots of the eigenvalues of $H_{\varphi}(\mathbf{0}_d)$ is a nonnegative real number. We thus take the positive square root of $H_{\varphi}(\mathbf{0}_d)$ and apply Corollary \ref{continuousamplitude} to get that
\begin{equation}
w_{(n, \lambda\xi)} = \frac{d^{\lambda\|\xi\|_1 + 1}}{(2\pi)^d}\cdot2\pi d^{2n}\frac{(2\pi\lambda)^{-d/2}}{\sqrt{(\|\xi\|_1/d)^d(d + 1)}}(1 + o(1))
\end{equation}
\noindent as $\lambda \to \infty.$
\end{proof}

\subsection{The Association Between $w_{(n, \xi)}$ and Modified Bessel Functions of the First Kind}

The cardinalities $w_{(n, \xi)}$ can be linked to products of modified Bessel functions of the first kind, which are functions given by the power series
\begin{equation}
I_{\nu}(z) = \sum_{n = 0}^{\infty} \frac{1}{n!\Gamma(n + \nu + 1)}\left(\frac{z}{2}\right)^{2n + \nu}, \hspace{1em} \text{for}\ \nu \in \C - \Z_{< 0},
\end{equation}
\noindent where $\Gamma$ is the eponymous gamma function. It is customary to keep to the principal branch of $(z/2)^{\nu}$ so that $I_{\nu}$ is analytic on $\C - (-\infty, 0]$ and two-valued and discontinuous on the cut $\text{Arg}\,z = \pm\pi$ (if $\nu \not\in \N_0$). To facilitate the asymptotic analysis of $w_{(n, m\mathbf{1}_d)},$ we will instead work with the normalized Bessel function
\begin{equation}
\tilde{I}_{\nu}(z) = \sum_{n = 0}^{\infty} \frac{\Gamma(\nu + 1)}{n!\Gamma(n + \nu + 1)}\left(\frac{z}{2}\right)^{2n}, \hspace{1em} \text{for}\ \nu \in \C - \Z_{< 0},
\end{equation}
\noindent which is entire on $\C.$ Next, observe that
\begin{equation}
w_{(n, \xi)} = (n + |\xi^+|)!(n + |\xi^-|)![x^n]\prod_{j = 1}^{d} \left(\frac{\tilde{I}_{|\xi_j|}\!\left(2\sqrt{x}\right)}{|\xi_j|!}\right).
\end{equation}
\noindent In particular,
\begin{equation}
w_{(n, m\mathbf{1}_d)} = \frac{n!(n + d|m|)!}{(|m|!)^d}[x^n]\left(\tilde{I}_{|m|}(2\sqrt{x})\right)^d.
\end{equation}
\noindent Bender, Brody, and Meister \cite{BBM} proved that
\begin{equation} \label{besselpower}
\left(\tilde{I}_{\nu}(z)\right)^d = \sum_{n = 0}^{\infty} \frac{\Gamma(\nu + 1)}{n!\Gamma(n + \nu + 1)}B^{(\nu)}_n(d)\left(\frac{z}{2}\right)^{2n},
\end{equation}
\noindent where $B^{(\nu)}_n(d)$ is a polynomial defined recursively by
\begin{equation}
B^{(\nu)}_n(d) = d\frac{\nu + n}{\nu + 1}B^{(\nu)}_{n - 1}(d) + \sum_{j = 1}^{n} \frac{b_j(\nu)}{n}\frac{\Gamma(\nu + 2)}{\Gamma(j + \nu + 2)}\binom{\nu + n}{j}B^{(\nu)}_{n - 1}(d).
\end{equation}
\noindent with initial condition $b^{(\nu)}_0(d) = 1.$ The sequence $\{b_n(\nu)\}_{n = 1}^{\infty}$ is identified by the generating function
\begin{equation}
\sum_{n = 1}^{\infty} \frac{b_n(\nu)}{(n - 1)!\Gamma(n + \nu + 1)}x^n = \frac{x}{\Gamma(\nu + 2)}\left(\frac{\sqrt{x}}{\nu + 1}\frac{I_{\nu}(2\sqrt{x})}{I_{\nu + 1}(2\sqrt{x})} - 2\right).
\end{equation}
\noindent However, Moll and Vignat \cite{MV} found an alternative characterization of $B^{(\nu)}_n(d)$ in terms of the exponent $d.$ To start, they exploit the Hadamard factorization
\begin{equation}
\tilde{I}_{\nu}(z) = \prod_{k = 1}^{\infty} \left(1 + \frac{z^2}{j_{\nu, k}^2}\right),
\end{equation}
\noindent where $\{j_{\nu, k}\}_{k = 1}^{\infty}$ is an enumeration of the zeros of $\tilde{I}_{\nu}(iz),$ from which it follows that
\begin{equation} \label{logbessel}
\log\tilde{I}_{\nu}(z) = \sum_{n = 1}^{\infty} \frac{(-1)^{n + 1}}{n}\zeta_n(2n)z^{2n},
\end{equation}
\noindent where $\zeta_{\nu}$ is the Bessel zeta function
\begin{equation}
\zeta_{\nu}(p) = \sum_{k = 1}^{\infty} \frac{1}{j_{\nu, k}^p}, \hspace{1em} \text{for}\ p > 1.
\end{equation}
\noindent They then cite the fact that the exponential of a power series is computed via
\begin{equation} \label{seriesexp}
\exp\left[\sum_{n = 1}^{\infty} a_n\frac{z^n}{n!}\right] = \sum_{n = 0}^{\infty} \mathbf{B}_n(a_1, \ldots, a_n)\frac{z^n}{n!},
\end{equation}
\noindent where $\mathbf{B}_n(a_1, \ldots, a_n)$ is the $n^{\text{th}}$ complete Bell polynomial. (Read \cite{Riordan}, section 5.2, for details.) Multiplying (\ref{logbessel}) by $d$ and then plugging it into (\ref{seriesexp}) leads to the following theorem.

\begin{theorem}[Moll, Vignat]
Define
\begin{equation}
a_n = a^{(\nu)}_n(d) = (-1)^{n - 1}(n - 1)!\zeta_{\nu}(2n)d.
\end{equation}
\noindent Then $B^{(\nu)}_n(d)$ is given by 
\begin{equation} \label{bellrecurrence}
B^{(\nu)}_n(d) = 4^n\frac{\Gamma(n + \nu + 1)}{\Gamma(\nu + 1)}\mathbf{B}_n(a^{(\nu)}_1(d), \ldots, a^{(\nu)}_n(d)).
\end{equation}
\end{theorem}

Substitute (\ref{bellrecurrence}) into (\ref{besselpower}) to get
\begin{equation}
w_{(n, m\mathbf{1}_d)} = \frac{4^n(n + d|m|)!}{(|m|!)^d}\mathbf{B}_n(a^{|m|}_1(d), \ldots, a^{|m|}_n(d)).
\end{equation}
\noindent We are now poised to derive the following asymptotic formula.

\begin{theorem}
\begin{equation}
w_{(n, m\mathbf{1}_d)} = \frac{\sqrt{2\pi}(|m|d)^{|m|d + 1/2}}{(|m|!e^{|m|})^d}\left(\frac{|m|d^2}{|m| + 1}\right)^n + O(d^{|m|d + 2n - 1/2})
\end{equation}
\noindent as $d \to \infty$ with $n$ and $m \neq 0$ fixed. If $m = 0,$ then
\begin{equation}
w_{(n, \mathbf{0}_d)} = n!d^n + O(d^{n - 1}).
\end{equation}
\end{theorem}

\begin{proof}
Since $\mathbf{B}_n(a^{|m|}_1(d), \ldots, a^{|m|}_n(d))$ is a polynomial in $d,$ it suffices to find its degree and leading coefficient. This is achieved by invoking the Bell polynomial's determinantal representation
\begin{equation} \label{determinant}
\mathbf{B}_n(a_1, \ldots, a_n) = \text{det}\left(\begin{matrix} a_1 & -1 & 0 & 0 & \cdots & 0 \\ a_2 & a_1 & -1 & 0 & \cdots & 0 \\ a_3 & \binom{2}{1}a_2 & a_1 & -1 & \ddots & 0 \\ \vdots & \vdots & \vdots & \ddots & \ddots & \vdots \\ a_{n - 1} & \binom{n - 2}{1}a_{n - 2} & \binom{n - 2}{2}a_{n - 3} & \binom{n - 2}{3}a_{n - 4} &\ddots & -1 \\ a_n & \binom{n - 1}{1}a_{n - 1} & \binom{n - 1}{2}a_{n - 2} & \binom{n - 1}{3}a_{n - 3} & \cdots & a_1 \end{matrix}\right).
\end{equation}
\noindent (This determinantal representation is explained in \cite{Collins}, albeit with a typographical error.) Each $a_i$ includes exactly one factor of $d,$ and so, because the above matrix is lower triangular save for the superdiagonal comprised entirely of $-1$s, $\mathbf{B}_n(a^{|m|}_1(d), \ldots, a^{|m|}_n(d))$ is an $n^{\text{th}}$ degree polynomial in $d$ with leading coefficient given by the product of the diagonal entries of (\ref{determinant}), which is
\begin{equation}
\left(a^{|m|}_1(d)\right)^n = (d\zeta_{|m|}(2))^n = \left(\frac{d}{4(|m| + 1)}\right)^n.
\end{equation}
\noindent (Here we have used the fact that $\zeta_{\nu}(2) = \frac{1}{4(\nu + 1)}$ which is elaborated on page 502 of \cite{Watson}.) Hence
\begin{equation} \label{lastasymptotic}
w_{(n, m\mathbf{1}_d)} = \frac{4^n(n + d|m|)!}{(|m|!)^d}\left(\left(\frac{d}{4(|m| + 1)}\right)^n + O(d^{n - 1})\right).
\end{equation}
\noindent Applying Stirling's approximation on $(n + d|m|)!$ completes the proof.
\end{proof}

Formula (\ref{lastasymptotic}) suggests that $w_{(n, m\mathbf{1}_d)}$ is roughly equal to
\[\binom{n + d|m|}{|m|, \ldots, |m|, n}\left(\frac{d}{|m| + 1}\right)^nn!\]
\noindent when $d$ is substantially large. In the case that $m = 0,$ we get $d^nn!,$ which is the number of ways to construct a length $n$ string $w$ and a permutation $w'$ of $w,$ provided that the characters of $w$ are always mutually distinct. This, of course, is untrue, and so $d^nn!$ is a bit of an overestimate for the number of abelian squares of length $2n.$ However, if one chooses a character from $\Sigma_d$ uniform randomly and independently for each letter of $w,$ then the probability that $w$ has no repeated characters is
\[\frac{d(d - 1)\cdots(d - n + 1)}{d^n} = 1 - O\!\left(\frac{1}{d^{n - 1}}\right).\]
\noindent The characters of $w$ are therefore asymptotically almost surely different, and so $d^nn!$ is a satisfactory estimate for $w_{(n, \mathbf{0}_d)}$ when $d$ is very large relative to $n.$

\section{Discussion and Future Work}
In this chapter we briefly discuss three possible avenues of future investigation.

\subsection{Fourier Analysis of Generating Functions}
We begin by applying the Cauchy-Hadamard formula to $W_{\xi},$ for $\xi \in \Z^d,$ to see that $W_{\xi}$ has a radius of convergence of
\begin{equation}
\frac{1}{\limsup\limits_{n \to \infty} |w_{(n, \xi)}|^{1/n}} = \frac{1}{\limsup\limits_{n \to \infty}\, (d^{2n + d/2 + \|\xi\|_1}(4\pi n)^{(1 - d)/2})^{1/n}} = \frac{1}{d^2}.
\end{equation}
\noindent This coupled with Plancharel's theorem from harmonic analysis leads us to conclude the following.

\begin{theorem} \label{spdunique}
The spectral density function $\mathcal{S}_{p_{1, d}^{\star}[x]},$ with $|x| < 1/d^2,$ is, up to sets of measure zero, the only operator $f : \R \rightarrow L^2(\T^d)$ that satisfies
\[x^{\frac{1}{2}\|\xi\|_1}W_{\xi}(x) = \widehat{f(x)}(\xi)\]
\noindent for all $x \in (-1/d^2, 1/d^2).$
\end{theorem}

Theorems \ref{fourierinversion} and \ref{spdunique} hint towards ``generatingfunctionological'' analogues of several classical results from harmonic analysis. We plan to explore these ideas further in a sequel paper.

\subsection{Stabilized Symmetric Polynomials}

Because the ring of symmetric polynomials with rational coefficients equals the rational polynomial ring generated by the power sum symmetric polynomials, one can then ``stabilize'' any rational symmetric polynomial. How, if at all, does this translate into operations on free objects over a finite set in general?

\subsection{A Unified Framework for ``Offsetness''}

At heart, the extent of a word pair's ``offsetness'' is the magnitude of the difference between the component words' ``frequencies,'' which are given by their Parikh vectors. Fundamentally, this is what suggests Fourier analytic treatment. Unitarity of the Fourier transform assures a one-to-one correspondence between the component Parikh vectors of a word pair and the offsetness of their concatenation.

Similar constructs should therefore be permissible on any Cartesian product of an unlabelled combinatorial class with itself. Candidates for analogous notions of offsetness include concave and convex compositions, rooted trees where the root has two children, and directed multigraphs. A more thorough, unified development of this concept may therefore be worthwhile.

\section*{Acknowledgments}

This collaboration evolved from a doctoral thesis project assigned to the second author by his thesis advisor, Hugo Woerdeman. We thank Dr. Woerdeman for his guidance and feedback. We also thank Robert Boyer, Eric Schmutz, and Robin Pemantle for engaging us in worthwhile technical discussions about certain aspects of this paper. The first author also thanks the conference organizers of the 2016 MAAGC workshop and Permutation Patterns 2016 for accepting our abstracts for poster sessions while this project was very much still a work in progress. We also thank the participants of those conferences who were interested in our poster. Their questions helped us decide how to motivate this subject.

\bibliographystyle{plain}

\end{document}